\documentclass[a4paper,12pt,reqno]{amsart}
\usepackage{graphics}
\usepackage{amsmath,amssymb,amsthm,graphicx,verbatim,psfrag,multirow}
\usepackage[colorlinks=true]{hyperref}
\usepackage{color}
\usepackage{ulem}
\graphicspath{{figs/}}

\theoremstyle{plain}
\newtheorem{theo}{Theorem}
\newtheorem{lem}[theo]{Lemma}
\newtheorem{con}[theo]{Conjecture}
\newtheorem{cor}[theo]{Corollary}

\newcommand{\inv}{\operatorname{inv}}

\theoremstyle{definition}

\newtheorem{rem}{Remark}[section]

\newcommand{\mt}{\operatorname{MT}}
\newcommand{\e}{\operatorname{E}}
\newcommand{\fd}{\overline{\Delta}}
\newcommand{\bd}{\underline{\Delta}}
\newcommand{\id}{\operatorname{Id}}
\newcommand{\st}{\operatorname{Strict}}
\newcommand{\gt}{\operatorname{GT}}
\newcommand{\ar}{\operatorname{AR}}
\newcommand{\ct}{\operatorname{CT}}
\newcommand{\sgn}{\operatorname{sgn}}

\numberwithin{equation}{section}

\begin{document}

\title[Gog and Magog trapezoids]{Constant term formulas for refined enumerations of Gog and Magog trapezoids}

\author[Ilse Fischer]{Ilse Fischer}
\address{Ilse Fischer, Fakult\"{a}t f\"{u}r Mathematik, Universit\"{a}t Wien, Oskar-Morgenstern-Platz 1, 1090 Wien, Austria}
\email{ilse.fischer@univie.ac.at}

\thanks{The author acknowledges support from the Austrian Science Foundation FWF, START grant Y463 and SFB grant F50.}

\begin{abstract}
Gog and Magog trapezoids are certain arrays of positive integers that generalize alternating sign matrices (ASMs)
and totally symmetric self-complementary plane partitions (TSSCPPs)
respectively. Zeilberger used constant term formulas to prove that there is the same number of $(n,k)$-Gog trapezoids as there is of $(n,k)$-Magog trapezoids, thereby providing so far the only proof for a weak version of a conjecture by Mills, Robbins and Rumsey from 1986. About $20$ years ago, Krattenthaler generalized Gog and Magog trapezoids and formulated an extension of their conjecture, and, recently, Biane and Cheballah generalized Gog trapezoids further and formulated a related conjecture. In this paper, we derive constant term formulas for various refined enumerations of generalized Gog trapezoids including those considered by Krattenthaler and by Biane and Cheballah. For this purpose we employ a result on the enumeration of truncated monotone triangles which is in turn based in the author's operator formula for the number of monotone triangles with prescribed bottom row. As a byproduct, we also generalize the operator formula for monotone triangles by including the inversion number and the complementary inversion number for ASMs. Constant term formulas as well as determinant formulas for the refined Magog trapezoid numbers that appear in Krattenthaler's conjecture are also deduced by using the classical approach based on non-intersecting lattice paths and 
the Lindstr\"om-Gessel-Viennot theorem. Finally, we review and partly extend a few existing tools that may be helpful in relating constant term formulas for Gogs to those for Magogs to eventually prove the above mentioned conjectures. 
\end{abstract}

\maketitle

\section{Introduction}

When Robbins and Rumsey \cite{RR86} introduced \emph{alternating sign matrices} in the 1980s, this was exciting for several reasons. On the one hand, they formulated a conjecture together with Mills  \cite{MilRobRum82} which states that the number of $n \times n$ ASMs is given by the compelling simple product formula 
$\prod\limits_{i=0}^{n-1} \frac{(3i+1)!}{(n+i)!}$. This conjecture turned out to be difficult to prove \cite{ZeilbergerASMProof, KuperbergASMProof} and still deserves to be better understood. On the other hand, it was soon realized that the product formula has appeared before as the counting formula for two classes of plane partitions, namely for \emph{totally symmetric self-complementary plane partitions} in an $2n \times 2n \times 2n$ box \cite{MilRobRum86} and for
\emph{descending plane partitions} whose parts do not exceed $n$ \cite{DPPMRR}. This gave rise to the search for explicit bijections between the three classes of objects to prove these facts directly. However, up to this day, no one was able to provide these bijections, and to find them is for many people the most important open problem in this field. 

As part of their attempt to provide a bijection between ASMs and TSSCPPs, Mills, Robbins and Rumsey \cite[Conjecture~7]{MilRobRum86} introduced two new classes of objects, one of which generalizes ASMs, while the other generalizes TSSCPPs and they conjectured equinumeracy between the two classes.
In his proof of the ASM counting formula, Zeilberger \cite{ZeilbergerASMProof} proved their conjecture. Following Zeilberger, the two types of objects are since then called \emph{$(n,k)$-Gog trapezoids} and \emph{$(n,k)$-Magog trapezoids} $n$ and $k$ being arbitrary non-negative integers.  Bijections exist only for the cases $k=1,2$, see \cite{cheballah1, cheballah2, betti}. 
Already Mills, Robbins and Rumsey introduced a pair of statistics on Gog trapezoids as well as a pair of statistics on Magog trapezoids, and they conjectured that these pairs have the same joint distribution \cite[Conjecture~7']{MilRobRum86}.
In 1996, Krattenthaler \cite{KrattGogMagog,krattstanley} generalized Gog trapezoids and Magog trapezoids by introducing a third parameter $m$ and formulated an extension of the (strong version of the) conjecture by Mills, Robbins and Rumsey. Recently, Biane and Cheballah introduced \emph{Gog pentagons} and \emph{GOGAm} pentagons, for which they also conjecture equinumeracy, and they have provided a bijective proof of their conjecture in the special cases when $k=1,2$.
 
In this article, we deduce constant term formulas for refined enumerations of Gogs including those that appear in the conjectures of Krattenthaler and of Biane and Cheballah. The derivation is based on the main result of \cite{FischerRefEnumASM} which provides a formula for the number of \emph{truncated monotone triangles}. This result is a generalization of the operator formula for the number of monotone triangles with prescribed bottom row \cite{FischerNumberOfMT}. In fact, in order to include as many refinement parameters as possible, we extend the main results from \cite{FischerRefEnumASM} and from  \cite{FischerNumberOfMT} by introducing the inversion number as well as the complementary inversion number of ASMs, see Section \ref{prel}. In order to be able to compare Gogs to Magogs, we also derive constant term formulas for refined enumerations of Magog trapezoids, see Section~\ref{magog}. In the final section, we mention (and partly extend) some known tools that may be helpful to relate constant term formulas for Gogs to those for Magogs.  

In general, the purpose of our results is twofold: On the one hand, they provide a framework for computational proofs of the conjectures of Krattenthaler, and of Biane and Cheballah, but, on the other hand, they can also be useful in the (experimental) search for statistics on the two classes of objects that have the same distribution to eventually construct an explicit bijection between Gogs and Magogs. In particular, the presentation aims at illustrating a certain versatility of the result on truncated monotone triangles from \cite{FischerRefEnumASM} and possibly indicates how to include other statistics on Gogs when it turns out to be useful. 

Before we present the preliminaries in Section~\ref{prel}, which are then applied in Section~\ref{app} to derive the constant term formulas for Gogs, we recall the definition of Gog trapezoids and Magog trapezoids, and state Krattenthaler's extension of 
\cite[Conjecture~7']{MilRobRum86}.

\subsection*{Gog trapezoids, Magog trapezoids and Krattenthaler's conjecture}

\subsubsection*{Gog trapezoids} In \cite{KrattGogMagog,krattstanley}, Krattenthaler 
generalizes Zeilberger's Gog trapezoids \cite{ZeilbergerASMProof} (which appeared unnamed for the first time in \cite[Conjecture~7]{MilRobRum86}) as follows: For $m,n,k$ non-negative integers with $k \le n$, an  \emph{$(m,n,k)$-Gog trapezoid} is defined as an arrangement of positive integers of the following form (the bullets correspond to the integers) 
\begin{equation*}
\begin{array}{ccccccccccc}
& & & & & & \bullet & & & &  \\
& & & & & \bullet & & \bullet & & &  \\
& & & & \bullet & & \bullet & & \bullet & &  \\
& & & \bullet & & \bullet & & \bullet & & \bullet &  \\
& & \bullet & & \bullet & & \bullet & & \bullet & & \bullet \\
& \bullet & & \bullet & & \bullet & & \bullet & & \bullet & \\
\bullet & & \bullet & & \bullet & & \bullet & & \bullet & &
\end{array},
\end{equation*}
where $n$ is the number of rows ($n=7$ in the example), $k$ is 
the number of north-east diagonals ($k=5$ in the example) and the entries in the $i$-th south-east diagonal ($1 \le i \le n$, counted from the left)  are bounded from above by 
$m+i$ such that 
\begin{itemize}
\item the entries are weakly increasing along NE-diagonals and 
SE-diagonals, and 
\item strictly increasing along rows. 
\end{itemize}
(By the strict increase along rows, it suffices to require that the entries in the $i$-th SE-diagonal are bounded from above by $m+i$ for $k \le i \le n$.) A $(3,7,5)$-Gog trapezoid is given next.
$$
\begin{array}{ccccccccccc}
& & & & & & 4 & & & &  \\
& & & & & 3 & & 5 & & &  \\
& & & & 2 & & 5 & & 8 & &  \\
& & & 2 & & 4 & & 7 & & 9 &  \\
& & {\color{red} 1} & & 4 & & 6 & & 7 & & {\color{blue} 10} \\
& {\color{red} 1} & & 4 & & 5 & & 7 & & 8 & \\
{\color{red} 1} & & 3 & & 5 & & 6 & & {\color{blue} 8} & &
\end{array}
$$
Zeilberger's original $n \times k$-Gog trapezoids are just $(0,n,k)$-Gog trapezoids as defined by Krattenthaler.
According to \cite{MilRobRum86,KrattGogMagog,krattstanley}, we define minima and maxima of Gog trapezoids as follows:
\begin{itemize}
\item  A \emph{minimum} is an entry equal to $1$. (The minima are located at the 
 bottom of the leftmost NE-diagonal.)
\item  A \emph{maximum} is an entry in the $k$-th NE-diagonal that is equal to the upper bound for the entries in its SE-diagonal. 
\end{itemize}
In the example, we have $3$ minima (indicated in red) and $2$ maxima (indicated in blue).

\subsubsection*{Magog trapezoids} In \cite{KrattGogMagog,krattstanley}, Krattenthaler 
generalized also Zeilberger's Magog trapezoids \cite{ZeilbergerASMProof} (the latter also appeared essentially for the first time in \cite[Conjecture~7]{MilRobRum86}): For $m,n,k$ non-negative integers with $k \le n$, an  \emph{$(m,n,k)$-Magog trapezoid} is defined as an arrangement of positive integers of the following form 
\begin{equation*}
\begin{array}{ccccccccccccccc}
& & & & & & \bullet & & & & &&&& \\
& & & & & \bullet & & \bullet & & & &&&&  \\
& & & & \bullet & & \bullet & & \bullet & & &&&& \\
& & & \bullet & & \bullet & & \bullet & & \bullet & &&&& \\
& & \bullet & & \bullet & & \bullet & & \bullet & & \bullet &&&&  \\
&  & & \bullet & & \bullet & & \bullet & & \bullet & & \bullet & & \\
 & &  & & \bullet & & \bullet & & \bullet & & \bullet & & \bullet
\end{array},
\end{equation*}
where $n$ is the number of rows ($n=7$ in the example), $k$ is 
the number of SE-diagonals ($k=5$ in the example) and the entries in the $i$-th NE-diagonal ($1 \le i \le n$, counted from the left)  are bounded from above by $m+i$ such that the entries are weakly increasing along NE-diagonals and SE-diagonals. Next we display a $(2,7,5)$-Magog trapezoid.
\begin{equation}
\label{magogexample}
\begin{array}{ccccccccccccccc}
& & & & & & \color{blue} 3 & & & & &&&& \\
& & & & & 2 & & 3 & & & &&&&  \\
& & & & 2 & & 3 & & 4 & & &&&& \\
& & & 1 & & 3 & & 4  & & 5 & &&&& \\
& & \color{red} 1 & & 2 & & 3 & & 5 & & \color{blue} 7 &&&&  \\
&  & & \color{red} 1 & & 3 & & 4 & & 6 & & 7 & & \\
 & &  & & 2 & & 3 & & 4 & & 7 & & \color{blue} 9
\end{array}
\end{equation}
Zeilberger's $n \times k$-Magog trapezoids are $(0,n,k)$-Magog trapezoids as defined by Krattenthaler.
According to \cite{MilRobRum86,KrattGogMagog,krattstanley}, we define minima and maxima of Magog 
trapezoids as follows:
\begin{itemize}
\item A \emph{minimum} is an entry equal to $1$ that is located in the leftmost SE-diagonal. 
\item A \emph{maximum} is an entry in the rightmost SE-diagonal that is equal to the upper bound for the entries in its NE-diagonal. 
\end{itemize}
In the example, we have $2$ minima (indicated in red) and $3$ maxima (indicated in blue).

Krattenthaler \cite{KrattGogMagog,krattstanley} conjectures the following relation between Gog trapezoids and Magog trapezoids. 

\begin{con} 
\label{conj_christ}
The number of $(m,n,k)$-Gog trapezoids with $p$ minima and $q$ maxima is equal to the number of $(m,n,k)$-Magog trapezoids with $p$ maxima and $q$ minima.
\end{con}

The case $m=0$ is Conjecture~7' in \cite{MilRobRum86}.
Conjecture~\ref{conj_christ} implies in particular that the number of $(m,n,k)$-Gog trapezoids is equal to the number of $(m,n,k)$-Magog trapezoids. In \cite{ZeilbergerASMProof}, Zeiberger proved this for $m=0$.

\section{Preliminaries: monotone triangles and truncated monotone triangles}
\label{prel}

\subsection{Inversion numbers}
The \emph{inversion number} of an ASM $A=(a_{i,j})_{1 \le i, j \le n}$, as defined in \cite[Eq. (18)]{RR86}, is given by 
$$
\inv(A) = \sum_{1 \le i'<i \le n, 1 \le j' \le j \le n} a_{i',j} a_{i,j'}.
$$
It generalizes the inversion number of permutations, that is, for a permutation matrix $A$, $\inv(A)$ is the number of 
inversions of the permutation $\pi=(\pi_1,\ldots,\pi_n)$, where $\pi_i$ is the column of the unique $1$ in row $i$.
Define the \emph{complementary inversion number} $\inv'(A)$ as follows.
$$
\inv'(A) = \sum_{1 \le i'<i \le n, 1 \le j \le j' \le n} a_{i',j} a_{i,j'}
$$
If $A$ is the permutation matrix of $\pi=(\pi_1,\ldots,\pi_n)$, then 
$\inv'(A)$ is clearly the number of inversions of $(\pi_n,\pi_{n-1},\ldots,\pi_1)$, and so we have $\inv(A)+\inv'(A)=\binom{n}{2}$ for permutation matrices.
 
It is a well-known fact \cite{BressoudPandC} that $n \times n$ alternating sign matrices correspond to \emph{Gelfand-Tsetlin patterns} with strictly increasing rows and bottom row 
$1,\ldots,n$. Recall that a Gelfand-Tsetlin pattern is a triangular array 
$(m_{i,j})_{1 \le j \le i \le n}$ of integers, where the elements are usually arranged as follows 
\begin{equation}
\label{indexing}
\begin{array}{ccccccccccccccccc}
  &   &   &   &   &   &   &   & m_{1,1} &   &   &   &   &   &   &   & \\
  &   &   &   &   &   &   & m_{2,1} &   & m_{2,2} &   &   &   &   &   &   & \\
  &   &   &   &   &   & \dots &   & \dots &   & \dots &   &   &   &   &   & \\
  &   &   &   &   & m_{n-2,1} &   & \dots &   & \dots &   & m_{n-2,n-2} &   &   &   &   & \\
  &   &   &   & m_{n-1,1} &   & m_{n-1,2} &  &   \dots &   & \dots   &  & m_{n-1,n-1}  &   &   &   & \\
  &   &   & m_{n,1} &   & m_{n,2} &   & m_{n,3} &   & \dots &   & \dots &   & m_{n,n} &   &   &
\end{array}
\end{equation}
such that there is a weak increase in northeast and southeast direction, i.e., 
$m_{i+1,j} \le m_{i,j} \le m_{i+1,j+1}$ for all $i,j$ with $1 \le j \le i < n$. 
Throughout this paper, we use the indexing of the entries of triangular arrays and partial triangular arrays as given here. A Gelfand-Tsetlin pattern in which each row is strictly increasing except for possibly the bottom row is said to be a \emph{monotone triangle}. (This definition deviates from the standard definition where also the bottom row needs to be strictly increasing.) Monotone triangles of order $n$ where also the bottom row is increasing and the entries are positive integers no greater than $m+n$ are just $(m,n,n)$-Gog trapezoids.

Suppose $M=(m_{i,j})_{1 \le j \le i \le n}$ is the monotone triangle corresponding to the ASM $A=(a_{i,j})_{1 \le i,j \le n}$, then it is not hard to see that
$$
\inv(A)= \# \{ (i,j): m_{i,j}=m_{i+1,j+1} \} \quad \text{and} \quad
\inv'(A) = \# \{ (i,j): m_{i,j}=m_{i+1,j} \}.  
$$
Since the $-1$'s of an $n \times n$ ASM $A$ correspond to the entries $m_{i,j}$ of the associated monotone triangle with $i<n$ and $m_{i+1,j} < m_{i,j} < m_{i+1,j+1}$, it follows that 
\begin{equation}
\label{m1}
\inv(A) + \inv'(A) = \binom{n}{2} - \text{($\#$ of $-1$'s in $A$)}.
\end{equation}
We use this to extend the definition of the inversion number and the complementary inversion number to all monotone triangles, that is, 
$$
\inv(M) = \# \{ (i,j): m_{i,j}=m_{i+1,j+1} \} \quad \text{and} \quad 
\inv'(M) = \# \{ (i,j): m_{i,j}=m_{i+1,j}\}.
$$
In \cite[Corollary 3.1]{FischerASMProof2015}, it was shown that the number of monotone triangles with bottom row $b_1,\ldots,b_n$ is the constant term of
\begin{equation}
\label{constilse2}
\prod_{i=1}^{n} (1+x_i)^{b_i} x_i^{-n+1} 
\prod_{1 \le i < j \le n} (x_i-x_j)(1+x_j+x_i x_j).
\end{equation}
(This derivation was based on \cite{FischerNumberOfMT}.)
In this section, we aim at generalizing this in order to obtain a constant term expression for the generating function of monotone triangles with respect to 
the two inversion numbers. For an increasing sequence $\mathbf{b}=(b_1,\ldots,b_n)$, we define the generating function
$$
\mt_{\mathbf{b}}(u,v) = \sum_{M} u^{\inv(M)} v^{\inv'(M)},
$$
where the sum is over all monotone triangles with bottom row $\mathbf{b}$.

We say that the two non-decreasing sequences $\mathbf{b}=(b_1,\ldots,b_n)$ and 
$\mathbf{a}=(a_1,\ldots,a_{n-1})$ are \emph{interlacing} (in symbols: $\mathbf{a} \prec \mathbf{b}$), if 
$$
b_1 \le a_1 \le b_2 \le a_2 \le \ldots \le a_{n-1} \le b_n. 
$$  
Consecutive rows of monotone triangles are obviously interlacing sequences, 
and we can write down the following recursion for $\mt_{\mathbf{b}}(u,v)$:
$$
\mt_{\mathbf{b}}(u,v) = \sum_{\mathbf{a} \prec \mathbf{b} \atop \mathbf{a}\text{\,strictly increasing}} \mt_{\mathbf{a}}(u,v) u^{\# \{i \, : \, a_i=b_{i+1}\}} v^{\# \{ i \, : \, a_i=b_i\}}
$$ 

We fix some notation that is needed in the following: We use the \emph{shift operator} $\e_x$, the \emph{forward difference} $\fd_{x}$ and the \emph{backward difference} $\bd_{x}$, which are defined as follows.
\begin{align*}
\e_x p(x) &:= p(x+1) \\
\fd_{x} &:= \e_x - \id \\
\bd_{x} &:= \id - \e_x^{-1}
\end{align*} 
We also need the following operator.
$$
\st_{x,y} = \e_{x}^{-1} \e_y  + u \, \bd_{x} \e_{y} - v \, \e_{x}^{-1} \fd_{y} = 
u \e_y + v \e_{x}^{-1} + (1 - u - v) \e_{x}^{-1} \e_y 
$$
Moreover, we need to work with the following extended definition of the 
summation
$$
\sum_{i=a}^{b} f(i) = \begin{cases} f(a) + f(a+1) + \ldots + f(b), & a \le b, \\  0, & b=a-1, \\
  -f(b+1) - f(b+2) - \ldots - f(a-1), & b+1 \le a-1 \end{cases}.
$$
The crucial property of the operator is the following: (Here we use the \emph{Iversion bracket}, i.e., $[\text{statement}]=1$ if the statement is true and $[\text{statement}]=0$ otherwise.)
$$
\sum_{(a_{i-1},a_i) \prec (b_{i-1},b_i,b_{i+1}) \atop a_{i-1} < a_i} f(a_{i-1},a_i) 
u^{[a_{i-1}=b_i]} v^{[a_i=b_i]} =
\left. \left[ \st_{b_i^{(1)},b_i^{(2)}} \sum_{a_{i-1}=b_{i-1}}^{b_i^{(1)}} \sum_{a_i=b_i^{(2)}}^{b_{i+1}} 
f(a_{i-1},a_i) \right] \right|_{b_i^{(1)}=b_i^{(2)}=b_i},
$$
which is true provided that $b_{i-1} < b_i < b_{i+1}$.  In case $u=v=1$, it is also true if 
$b_{i-1} \le  b_i \le b_{i+1}$.
However, note that also 
\begin{equation}
\label{left}
\sum_{a_i=b_i}^{b_{i+1}} f(a_i) v^{[a_i=b_i]} = \left. \left[ \st_{b_i^{(1)},b_i^{(2)}} \sum_{a_i=b_{i}^{(2)}}^{b_{i+1}} f(a_i) \right] \right|_{b_i^{(1)}=b_i^{(2)}=b_i}
\end{equation}
and 
\begin{equation}
\label{right}
\sum_{a_i=b_i}^{b_{i+1}} f(a_i) u^{[a_i=b_{i+1}]} = \left. \left[ \st_{b_{i+1}^{(1)},b_{i+1}^{(2)}} \sum_{a_i=b_{i}}^{b_{i+1}^{(1)}} f(a_i) \right] \right|_{b_{i+1}^{(1)}=b_{i+1}^{(2)}=b_{i+1}}
\end{equation}
whenever $b_i \le b_{i+1}$. 
This implies 
\begin{multline*}
 \sum_{\mathbf{a} \prec \mathbf{b} \atop \mathbf{a} \text{\,strictly increasing} } f(\mathbf{a}) u^{\# \{i \, : \, a_i=b_{i+1}\}} v^{\# \{ i \, : \, a_i=b_i\}} \\
= \left. \left[ \st_{b_1^{(1)},b_1^{(2)}} \dots  \st_{b_n^{(1)},b_n^{(2)}} 
\sum_{a_1=b_1^{(2)}}^{b_2^{(1)}} \sum_{a_2=b_2^{(2)}}^{b_3^{(1)}} \cdots
\sum_{a_{n-1}=b_{n-1}^{(2)}}^{b_n^{(1)}} f(\mathbf{a}) \right] \right|_{b_i^{(1)}=b_i^{(2)}=b_i}
\end{multline*}
whenever $b_1 < b_2 < \ldots < b_n$, and, under the assumption that  $u=v=1$, it is sufficient to require $b_1 \le b_2 \le \ldots \le b_n$.
For $\mathbf{x}=(x_1,\ldots,x_n)$, we define 
$$
\gt_{n}(\mathbf{x}) = \prod_{1 \le i < j \le n} \frac{x_j-x_i+j-i}{j-i} \quad 
\text{and} \quad M_n(u,v,\mathbf{x}) = \prod_{1 \le p<q \le n} \st_{x_q,x_p} \gt_n(\mathbf{x}).
$$
Theorem~2.1 in \cite{FischerASMProof2015} can be generalized as follows.
\begin{theo}
\label{operator}
Suppose $\mathbf{b}=(b_1,\ldots,b_n)$ is a strictly  increasing sequence of integers, then the generating function of monotone triangles with bottom row $b_1,\ldots,b_n$ w.r.t.\ the two inversion numbers is the evaluation of the polynomial $M_n(u,v,\mathbf{x})$ at $(x_1,\ldots,x_n)=(b_1,\ldots,b_n)$. When considering the special case $u=v=1$, it suffices to require that $(b_1,\ldots,b_n)$ is weakly increasing.
\end{theo}
The proof is analogous to the proof of Theorem~2.1 in  \cite{FischerASMProof2015}. Moreover, using the fact that 
$$
\st_{x,y} = \e_{x}^{-1} (\id + u \, \fd_{x} + (1-v) \, \fd_{y} + u \, \fd_{x} \fd_{y}),
$$
the following generalization of Corollary~3.1 in \cite{FischerASMProof2015} can be proved.
\begin{cor}
\label{constant}
Suppose $\mathbf{b}=(b_1,\ldots,b_n)$ is a strictly increasing sequence of integers, then the generating function of monotone triangles with bottom row $b_1,\ldots,b_n$ w.r.t.\ the two inversion numbers is the constant term in $x_1,\ldots,x_n$ of the following Laurent polynomial.
$$
\prod_{i=1}^{n} (1+x_i)^{b_i} x_i^{-n+1} \prod_{1 \le i < j \le n} (x_i-x_j)(1+(1-v)x_i+
u (x_j + x_i x_j))
$$
When considering the special case $u=v=1$, it suffices to require that $(b_1,\ldots,b_n)$ is weakly increasing.
\end{cor}

In Appendix~\ref{v1minusu}, we elaborate on the case $v=1-u$, which is an interesting case since there exist certain combinatorial tools to handle it. Also note that we can recover the generating function with respect to number of $-1$ in the ASM proved in \cite{FischerSimplifiedProof} by setting $u=v=1/Q$ and multiplying with $Q^{\binom{n}{2}}$. This follows from \eqref{m1}. 

\begin{rem} There exist two alternative constant term expressions for the number
of monotone triangles with bottom row $\mathbf{b}=(b_1,\ldots,b_n)$ that could replace \eqref{constilse2} in all what follows in the special case $u=v=1$. Both of them are also based on 
\cite{FischerNumberOfMT}.
First, it was shown in 
\cite[Theorem~3]{FischerOpFormulaVSASM} that the constant term of
\begin{equation}
\label{alternative}
\prod_{i=1}^{n} x_i^{-n+1-b_i} (1-x_i)^{-n} \prod_{1 \le i < j \le n} (x_j-x_i)(1-x_j+x_i x_j)
\end{equation}
is the number of monotone triangles with bottom row $\mathbf{b}=(b_1,\ldots,b_n)$ if $b_i \ge 0$ for all $i$.
Note that this is not a Laurent polynomial but a rational function, and so we need to clarify how we expand it into a Laurent series. Here and throughout the remainder of the paper, unless stated otherwise, we expand in powers with non-negative exponents, in 
particular, we have 
$$
(1-x_i)^{-n} = \sum_{j=0}^{\infty} \binom{-n}{j} (-1)^j x_i^j
$$
in \eqref{alternative}. In order to present the second alternative formula, we 
define ${\mathcal AS}_{x_1,\ldots,x_n}$ to be the antisymmetrizer w.r.t.\ $x_1,\ldots,x_n$, that is 
\begin{equation}
\label{antisym}
{\mathcal AS}_{x_1,\ldots,x_n} p(x_1,\ldots,x_n):= \sum_{ \sigma \in {\mathcal S}_n} \sgn \sigma \, p(x_{\sigma(1)},\ldots,x_{\sigma(n)}).
\end{equation}
In \cite{FischerRieglerVSASM} it was shown that 
the number of monotone triangles with bottom row $\mathbf{b}=(b_1,\ldots,b_n)$ is given by the constant term of the following expression.
\begin{equation}
\label{alternative2}
{\mathcal AS}_{x_1,\ldots,x_n} \left[ \prod_{i=1}^{n} (1+x_i)^{b_i} \prod_{1 \le i < j \le n} (1+ x_j + x_i x_j) \right] \prod_{1 \le i < j \le n} (x_j - x_i)^{-1}
\end{equation}
It can be seen directly that this constant term is equal to the constant term of \eqref{constilse2} as, for any polynomial $f(x_1,\ldots,x_n)$, the constant term of $f(x_1,\ldots,x_n) \prod\limits_{1 \le i < j \le n} (x_i^{-1} - x_j^{-1})$ is equal to the constant term of ${\mathcal AS}_{x_1,\ldots,x_n} \left[ f(x_1,\ldots,x_n) \right] / \prod\limits_{1 \le i < j \le n} (x_i-x_j)$.
Note that the expression in \eqref{alternative2} is a polynomial in $x_1,\ldots,x_n$ (and not only a Laurent polynomial). 

It should be noted that \eqref{alternative} and \eqref{alternative2} could actually be generalized to arbitrary $u$ and $v$.
\end{rem}

\subsection{Refined enumeration with respect to the top entry}

For a positive integer $n$, non-negative integers $\min \le a \le \max$
and $\mathbf{x}=(x_1,\ldots,x_n)$, we define
$$
\gt_{n,a,\min,\max}(\mathbf{x}) = \sum_{b=0}^{\max-\min} (-1)^{a+\min+b} \binom{b}{a-\min} \det_{1 \le i, j \le n} \left( \binom{x_i+i-\min-1}{j-1 + b [j=n]} \right).
$$
Suppose $\min \le b_1 \le b_2 \le \ldots \le b_n \le \max$ is a sequence of 
integers, then $\gt_{n,a,\min,\max}(b_1,\ldots,b_n)$ is the number of Gelfand-Tsetlin patterns with bottom row $b_1,\ldots,b_n$ and top row $a$. This follows, for instance, from a result in 
\cite[Theorem 3, $S=\emptyset,P=Q=1$]{FischerSimplifiedProof}, as   
$$
\prod_{\min \le i \le \max, i \not=a} \frac{x-i}{a-i} =
\sum_{b=0}^{\max-\min} (-1)^{a+\min+b} 
\binom{b}{a-\min} \binom{x-\min}{b}.
$$ 
Further, we define 
$$
M_{n,a,\min,\max}(u,v,\mathbf{x}) = \prod_{1 \le p<q \le n} \st_{x_q,x_p} \gt_{n,a,\min,\max}(\mathbf{x}),
$$
then it follows from \cite[Theorem~3]{FischerSimplifiedProof} and the ideas above that the evaluation of $M_{n,a,\min,\max}(u,v,\mathbf{x})$ at $\mathbf{x}=(b_1,\ldots,b_n)$ is the generating function of monotone triangles with bottom row $b_1,\ldots,b_n$ and top row $a$, provided that $\min \le b_1 < b_2 < \ldots < b_n \le \max$. The refined version of Corollary~\ref{constant} is then the following.

\begin{theo}
\label{top}
Suppose $\mathbf{b}=(b_1,\ldots,b_n)$ is a strictly increasing sequence of integers and $\min \le a \le \max$ are integers with $\min \le b_1$ and $b_n  \le \max$. Then the generating function of monotone triangles with bottom row $b_1,\ldots,b_n$ and top row $a$ is the constant term of the rational function one obtains by multiplifying the Laurent polynomial in Corollary~\ref{constant} with 
$$
\sum_{b=0}^{\max - \min} (-1)^{a + \min + b} \binom{b}{a-\min} \prod_{i=1}^{n} (1+x_i)^{-\min} \langle z^b \rangle \prod_{i=1}^{n} \frac{1}{1-x_i^{-1} z},
$$
where $\langle z^b \rangle$ means that we take the coefficient of $z^b$ and 
$\frac{1}{1-x_i^{-1} z} = \sum\limits_{k=0}^{\infty} z^k x_i^{-k}$.
Again, when considering the special case $u=v=1$, it suffices to require that 
$\mathbf{b}=(b_1,\ldots,b_n)$ is weakly increasing.
\end{theo}

\subsection{Truncated monotone triangles: $(\mathbf{s},\mathbf{t})$-trees} Next we turn to certain partial monotone triangles.
Let $l,r,n$ be non-negative integers with $l+r \le n$. Suppose 
$\mathbf{s}=(s_1,s_2,\ldots,s_l)$, $\mathbf{t}=(t_{n-r+1},t_{n-r+2},\ldots,t_n)$ are sequences of non-negative integers, where $\mathbf{s}$ is weakly decreasing, while $\mathbf{t}$ is weakly increasing. An $(\mathbf{s},\mathbf{t})$-tree of order $n$ is an integer array whose shape is obtained from the shape of a monotone triangle with $n$ rows when deleting the bottom $s_i$ entries from the $i$-th NE-diagonal for $1 \le i \le l$  (NE-diagonals are counted from the left) and the bottom $t_i$ entries from the $i$-th SE-diagonal for $n-r+1 \le i \le n$ (SE-diagonals are also counted from the left), see Figure~\ref{st_tree}. We assume in the following that 
there is no interference between the deletion of the entries in the $l$  
leftmost NE-diagonals (as prescribed by $\mathbf{s}$) with the deletion of the entries from the $r$ rightmost SE-diagonals (as prescribed by $\mathbf{t}$).

In an $(\mathbf{s},\mathbf{t})$-tree, an entry $m_{i,j}$ is said to be \emph{regular} if it has a SW neighbour $m_{i+1,j}$ and a SE neighbour $m_{i+1,j+1}$.
We require the following monotonicity properties in an 
$(\mathbf{s},\mathbf{t})$-tree: 
\begin{enumerate}
\item Each regular entry $m_{i,j}$ has to fulfill $m_{i+1,j} \le m_{i,j} \le m_{i+1,j+1}$.
\item Two adjacent regular entries $m_{i,j}, m_{i,j+1}$ in the same row have to be distinct.   
\end{enumerate}
This extends the notion of monotone triangles, as a monotone triangle of order $n$ is just an $(\mathbf{s},\mathbf{t})$-tree, where $l,r$ are any two numbers 
with $l+r \le n$, $\mathbf{s}=(\underbrace{0,\ldots,0}_{l})$ and $\mathbf{t}=(\underbrace{0,\ldots,0}_r)$.

\begin{figure}
\scalebox{0.3}{
\includegraphics{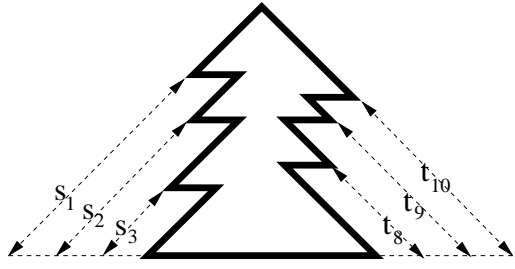}}
\caption{\label{st_tree} Shape of an $(\mathbf{s},\mathbf{t})$-tree.}
\end{figure}

We extend the definition of the two inversion numbers as follows. Suppose $M=(m_{i,j})$ (using the indexing indicated in \eqref{indexing}) is an $(\mathbf{s},\mathbf{t})$-tree of order $n$ and
define 
\begin{align*}
\inv(M) & = \# \{ m_{i,j} \text{ regular }: m_{i,j}=m_{i+1,j+1} \}, \\
\inv'(M) & = \# \{ m_{i,j}  \text{ regular }: m_{i,j}=m_{i+1,j}\}.
\end{align*}
We define further generalizations of the inversion numbers that are useful for our purposes: The generalization of $\inv$ depends on a subset $J \subseteq \{j | t_{j} < t_{j+1}\} \cup \{n\}$ of \emph{exceptional} SE-diagonals and is defined as 
\begin{multline*}
\inv_{J}(M) = \# \{ m_{i,j} \text{ regular }: m_{i,j}=m_{i+1,j+1} \\ \text{ and } m_{i+1,j+1} 
\text{ is not the bottom entry of an exceptional SE-diagonal} \}.
\end{multline*}
The generalization of the complementary inversion number $\inv'$ depends on a subset $I \subseteq \{1\} \cup \{i | s_{i-1} > s_{i} \}$ of exceptional NE-diagonals and is defined as 
\begin{multline*}
\inv'_{I}(M) = \# \{ m_{i,j} \text{ regular }: m_{i,j}=m_{i+1,j} \\ \text{ and } m_{i+1,j} 
\text{ is not the bottom entry of an exceptional NE-diagonal} \}.
\end{multline*}

Next we define generalizations of the forward difference operator and of the backward difference operator, namely the $v$-forward difference operator $\fd[v]_x$(w.r.t. the variable $x$) and the $u$-backward difference operator $\bd[u]_x$.
$$
\fd[v]_x = (1 + (1-v) \fd_x)^{-1} \fd_x \quad \text{and} \quad \bd[u]_x = (1+(u-1) \bd_x)^{-1} \bd_x
$$
These operators are well-defined in our context, since 
$$
(1 + (1-v) \fd_x)^{-1} = \sum_{i=0}^{\infty} (v-1)^{i} \fd_x^i \quad \text{and} \quad 
(1 + (u-1) \bd_x)^{-1} = \sum_{i=0}^{\infty} (1-u)^{i} \bd_x^i
$$
and the sums are finite if applied to polynomials, because for each polynomial $p(x)$ of degree $d$ we have 
$\fd_x^{i} p(x) = \bd_x^{i} p(x) = 0$ for all $i>d$. 
The following property of $\fd[v]_x$ is crucial: 
\begin{multline*}
 - \fd[v]_{b_1}  \left. \left[ \st_{b_1^{(1)},b_1^{(2)}} \dots  \st_{b_n^{(1)},b_n^{(2)}} 
\sum_{a_1=b_1^{(2)}}^{b_2^{(1)}} \sum_{a_2=b_2^{(2)}}^{b_3^{(1)}} \cdots
\sum_{a_{n-1}=b_{n-1}^{(2)}}^{b_n^{(1)}} f(\mathbf{a}) \right] \right|_{b_i^{(1)}=b_i^{(2)}=b_i} \\
= - \fd_{b_1} \left. \left[ \st_{b_2^{(1)},b_2^{(2)}} \dots  \st_{b_n^{(1)},b_n^{(2)}} 
\sum_{\raisebox{-1.2mm}{\tiny $a_1=b_1$}}^{b_2^{(1)}} \sum_{a_2=b_2^{(2)}}^{b_3^{(1)}} \cdots
\sum_{a_{n-1}=b_{n-1}^{(2)}}^{b_n^{(1)}} f(\mathbf{a}) \right] \right|_{b_i^{(1)}=b_i^{(2)}=b_i}  \\
= \left. \left[ \st_{b_2^{(1)},b_2^{(2)}} \dots  \st_{b_n^{(1)},b_n^{(2)}} 
\sum_{a_2=b_2^{(2)}}^{b_3^{(1)}} \cdots
\sum_{a_{n-1}=b_{n-1}^{(2)}}^{b_n^{(1)}} f(b_1,a_2,\ldots,a_{n-1}) \right] \right|_{b_i^{(1)}=b_i^{(2)}=b_i}
\end{multline*}
A similar identity is true for $\bd[u]_x$. These observations imply almost immediately the following result. The case $u=v=1$ appeared first in 
\cite{FischerRefEnumASM}, however the proof given there can be generalized easily, see also \cite{FischerASMProof2015}.

\begin{theo} Let $n,l,r$ be non-negative integers with $l+r \le n$. Suppose $b_1,\ldots,b_n$ is a strictly increasing sequence of integers, and $\mathbf{s}=(s_1,s_2,\ldots,s_l)$, $\mathbf{t}=(t_{n-r+1},t_{n-r+2},\ldots,t_n)$ are a
weakly decreasing and a weakly increasing sequence of non-negative integers, respectively.
Then the evaluation of the following polynomial
\begin{equation}
\label{stoperator}
(-\fd[v]_{x_1})^{s_1} \cdots (-\fd[v]_{x_l})^{s_l} 
\bd[u]_{x_{n-r+1}}^{t_{n-r+1}} 
  \cdots \bd[u]_{x_{n}}^{t_{n}}  
M_n(u,v,\mathbf{x})
\end{equation}
at $(x_1,\ldots,x_n)=(b_1,\ldots,b_n)$ 
is the generating function with respect to $\inv$ and $\inv'$ of $(\mathbf{s},\mathbf{t})$-trees of order $n$ with the following properties:
\begin{itemize}
\item For $1 \le i \le n-r$, the  bottom entry of the $i$-th NE-diagonal is $b_i$.
\item For $n-r+1 \le i \le n$, the bottom entry of the $i$-th SE-diagonal is $b_i$.
\end{itemize}
Furthermore, suppose  $I \subseteq \{1\} \cup \{i | s_{i-1} > s_{i} \}$ and $J \subseteq \{j | t_{j} < t_{j+1}\} \cup \{n\}$, then the generating function with respect to $\inv_{J}$ and $\inv'_{I}$ of these 
$(\mathbf{s},\mathbf{t})$-trees is obtained from \eqref{stoperator}  by applying
$$
\prod_{i \in I} \left( 1+(1-v) \fd_{x_i} \right)^{-1} \prod_{j \in J} \left( 1+(u-1) \bd_{x_j} \right)^{-1},
$$
and then evaluating at $(x_1,\ldots,x_n)=(b_1,\ldots,b_n)$. 
Finally, suppose $\min \le a \le \max$ are integers with $\min \le b_1$ and 
$b_n \le \max$, then we obtain the generating function of the above mentioned 
$(\mathbf{s},\mathbf{t})$-trees that have $a$ in the top row if 
$M_n(u,v,\mathbf{x})$ is replaced by $M_{n,a,\min,\max}(u,v,\mathbf{x})$.

When considering the special case $u=v=1$, all results are true also if we only require 
that $b_1,\ldots,b_n$ is weakly increasing.
\end{theo}

In order to prove the modification concerning the generalized inversion numbers $\inv_{J}$ and $\inv'_{I}$, it has to be noted that 
$$\st_{x,y} =1+(1-v) \fd_y$$
when applied to functions that are independent of $x$, while 
$$
\st_{x,y}=1+(u-1) \bd_x
$$ 
when applied to functions that are independent of $y$. Compare also to \eqref{left} and \eqref{right}.

We translate the formula in the theorem into a constant term expression generalizing the constant term expression of Corollary~\ref{constant}: To this end, we first observe that the application of $\fd_{b_i}$ to the constant term 
expression in Corollary~\ref{constant} corresponds to the multiplication with $(1+x_i)-1=x_i$, while the application of $\bd_{b_i}$ corresponds to the multiplication with $1-1/(1+x_i)=x_i/(1+x_i)$. This implies the following relations:
$$
\fd[u]_{b_i} \sim \frac{x_i}{1+(1-v) x_i} \quad \text{and} \quad 
\bd[u]_{b_i} \sim \frac{x_i}{1+ u x_i}.
$$

\begin{cor} 
\label{constant1}
The generating function of $(\mathbf{s},\mathbf{t})$-trees as described in the theorem is given by the constant term in $x_1,\ldots,x_n$ of the 
following expression:
\begin{multline*}
\prod_{i=1}^{l} \left( \frac{x_i}{(v-1) x_i -1} \right)^{s_i} \prod_{i=n-r+1}^{n} 
\left( \frac{x_i}{u x_i +1} \right)^{t_i} \prod_{i=1}^{n} (1+x_i)^{b_i} x_i^{-n+1} \\
\times 
\prod_{1 \le i < j \le n} (x_i-x_j)(1+(1-v) x_i + u (x_j + x_i x_j))
\end{multline*}
If $I,J$ are two sets as described in the theorem, then the generating function with respect to $\inv_{J},\inv'_{I}$ is the constant term in $x_1,\ldots,x_n$
of the expression one obtains by multiplying the rational function above with 
$$
\prod_{i \in I} \frac{1}{1+(1-v)x_i} \prod_{j \in J} \frac{1+x_j}{1+u x_j}.
$$
If $\min \le a \le \max$ are integers with $\min \le b_1$ and $b_n \le \max$, then the generating function of the above mentioned 
$(\mathbf{s},\mathbf{t})$-trees whose top row is $a$ is the constant term of the rational function which is obtained by multiplying the respective expression by $P_{n,a,\min,\max}(\mathbf{x})$.
\end{cor}

\section{Application: constant term formula for the generating function of $(m,n,k)$-Gog trapezoids}
\label{app}

Using the terminology of the previous section, $(m,n,k)$-Gog trapezoids with bottom row $b_1,\ldots,b_k$ correspond to $(\mathbf{s},\mathbf{t})$-trees of order $n$, where
$\mathbf{s}=\emptyset$, $\mathbf{t}=(0,1,\ldots,n-k-1)$, such that 
\begin{enumerate}
\item the bottom entries of the first $k$ NE-diagonals are $b_1,\ldots,b_k$,
\item the bottom entries of the last $n-k$ SE-diagonals are $m+k+1,\ldots,m+n$.
\end{enumerate}
In order to obtain the Gog trapezoid from the $(\mathbf{s},\mathbf{t})$-tree, 
one has to delete the $(k+1)$-st NE-diagonal from the $(\mathbf{s},\mathbf{t})$-tree. 

Corollary~\ref{constant1} now implies the following formulas.

\begin{theo}
\label{gogfirst}
The number of $(m,n,k)$-Gog trapezoids with bottom row $b_1,\ldots,b_k$ is the constant term of
$$
\prod_{i=1}^{k} (1+x_i)^{b_i} x_i^{-n+1} \prod_{i=k+1}^{n} (1+x_i)^{m+k+1} x_i^{-n+i-k}
\prod_{1 \le i < j \le n} (x_i-x_j)(1+x_j+x_i x_j).
$$
If $m=0$, then the bottom row is necessarily $1,2,\ldots,k$, and so
the number of $(0,n,k)$-Gog trapezoids is the constant term of 
$$
\prod_{i=1}^{k} (1+x_i)^{i} x_i^{-n+1} \prod_{i=k+1}^{n} (1+x_i)^{k+1} x_i^{-n+i-k}
\prod_{1 \le i < j \le n} (x_i-x_j)(1+x_j+x_i x_j).
$$
\end{theo}

We are interested in the refined enumeration of $(m,n,k)$-Gog trapezoids with respect to four parameters.

\subsection{Inversion numbers} The two inversion numbers are defined in the obvious way: Suppose $M=(m_{i,j})$ is an $(m,n,k)$-Gog trapezoid, then
\begin{align*}
\inv(M) & = \# \{ m_{i,j} : m_{i,j}=m_{i+1,j+1} \}, \\
\inv'(M) & = \# \{ m_{i,j} : m_{i,j}=m_{i+1,j}\}.
\end{align*}
So, in order for the entry $m_{i,j}$ to contribute to $\inv(A)$, it must have a SE neighbor in the $(m,n,k)$-Gog trapezoid and this neighbor must be equal to $m_{i,j}$; however, $m_{i,j}$ need not to have a SW neighbor. Similar for the complementary inversion number. Corollary~\ref{constant1} implies the following.

\begin{theo} 
\label{uv}
The generating function of $(m,n,k)$-Gog trapezoids with bottom row 
$b_1,\ldots,b_k$ w.r.t.\ $\inv$ and  $\inv'$ is the constant term of  
\begin{equation}
\label{gog1}
\prod_{i=1}^{k} (1+x_i)^{b_i} x_i^{-n+1} \prod_{i=k+1}^{n} (1+x_i)^{m+i+1} (1+ u x_i)^{k-i} x_i^{-n+i-k}
\prod_{1 \le i < j \le n} (x_i-x_j)(1+(1-v)x_i+u(x_j+x_i x_j))
\end{equation}
in $x_1,\ldots,x_n$.
\end{theo}
The additional factor $\prod\limits_{i=k+1}^{n} \frac{1+x_i}{1+ u x_i}$ is caused by the 
fact that the $u$-weight should not take into account the entries in the $k$-th
NE-diagonal that are equal to their upper bound.

\subsection{Minima and maxima} 
In order to involve the number of minima and the number of maxima in the Gog-trapezoids, we have to apply the version of Corollary~\ref{constant1} where the requirement on $b_1,b_2,\ldots,b_n$ is only that it is weakly increasing, and thus we cannot involve the inversion numbers at the same time.

An $(m,n,k)$-Gog trapezoids with $p$ \emph{minima} can be identified with $(\mathbf{s},\mathbf{t})$-trees of order $n$, where $\mathbf{s}=(p-1)$, $\mathbf{t}=(0,1,\ldots,n-k-1)$ and where the bottommost entry of the first NE-diagonal is set to $2$. In order to see this, delete in the $(m,n,k)$-Gog trapezoids with $p$ minima the occurences of $1$ except for the topmost which is replaced by $2$. In order to obtain a formula for the number we 
have to multiply \eqref{gog1} by $\left( -x_1 \right)^{p-1}$ and set $b_1=2$.  We obtain 
\begin{multline}
\label{gog2}
(-1)^{p-1} (1+x_1)^2 x_1^{p-n} 
\prod_{i=2}^{k} (1+x_i)^{b_i} x_i^{-n+1} \prod_{i=k+1}^{n} (1+x_i)^{m+k+1} x_i^{-n+i-k} \\
\times
\prod_{1 \le i < j \le n} (x_i-x_j)(1+x_j+x_i x_j).
\end{multline}
The case $p=0$ is not covered by this formula. In that case, we use 
\eqref{gog1} with the appropriate $b_1>1$.

Now let $M \subseteq \{k+1,\ldots,n\}$ and suppose we require to have a \emph{maximum} in the $i$-th SE-diagonal precisely if $i \in M$ (and possibly in the $k$-th SE-diagonal, in which case we have $b_k=m+k$).
Then 
we need to multiply \eqref{gog2} by 
$$
\prod_{i \in M} \frac{x_i}{1+x_i} \prod_{i \in \{k+1,k+2,\ldots,n\} \setminus M} 
\frac{1}{1+x_i}.
$$
If we sum over all subsets $M$ with cardinality $q$, we obtain 
$$
\prod_{i=k+1}^{n} \frac{1}{1+x_i} 
e_q\left(x_{k+1},\ldots,x_n\right),
$$
where $e_q$ denotes the $q$-th elementary symmetric function. As $e_q(x_{k+1},\ldots,x_n)$ is the coefficient of $Q^q$ in $\prod\limits_{i=k+1}^n (1+Q x_i)$, we obtain the following theorem.

\begin{theo}
\label{gog}
The generating function of $(m,n,k)$-Gog trapezoids with bottom row 
$1,b_2,\ldots,b_k$ w.r.t.\ the weight
$$
P^{\# \text{of minima}} \, Q^{\# \text{of maxima not including $b_k$}}
$$
is the constant term of 
\begin{multline}
\label{gog3}
P (1+P x_1)^{-1} (1+x_1)^2 x_1^{-n+1} 
\prod_{i=2}^{k} (1+x_i)^{b_i} x_i^{-n+1}  \prod_{i=k+1}^{n} \left( 1+ Q x_i \right) (1+x_i)^{m+k} x_i^{-n+i-k} \\
\times
\prod_{1 \le i < j \le n} (x_i-x_j)(1+x_j+x_i x_j)
\end{multline}
in $x_1,\ldots,x_n$, while the generating function of $(m,n,k)$-Gog trapezoids with 
bottom row $b_1,b_2,\ldots,b_k$ w.r.t.\ the weight
$$
Q^{\# \text{of maxima not including $b_k$}}
$$
is the 
constant term of
\begin{equation}
\label{gog4}
\prod_{i=1}^{k} (1+x_i)^{b_i} x_i^{-n+1}  \prod_{i=k+1}^{n} \left( 1+ Q x_i \right) (1+x_i)^{m+k} x_i^{-n+i-k} 
\prod_{1 \le i < j \le n} (x_i-x_j)(1+x_j+x_i x_j)
\end{equation}
in $x_1,\ldots,x_n$.
\end{theo}

The generating functions in Theorems~\ref{uv} and \ref{gog} can be restricted further to Gog trapezoids with fixed top entry by multiplying with the Laurent series provided in Theorem~\ref{top}.

In case $m=0$, there is always a minimum and the bottom row is forced. Also note that the last entry of the bottom row is a maximum in this case.

\begin{cor} The generating function of $(0,n,k)$-Gog trapezoids w.r.t. the number of minima and maxima is the constant term of
\begin{multline*} 
P Q (1+P x_1)^{-1} (1+x_1)^2 x_1^{-n+1} 
\prod_{i=2}^{k} (1+x_i)^{i} x_i^{-n+1} \prod_{i=k+1}^{n} \left( 1+ Q x_i \right) (1+x_i)^{m+k} x_i^{-n+i-k} \\
\times
\prod_{1 \le i < j \le n} (x_i-x_j)(1+x_j+x_i x_j).
\end{multline*}
\end{cor}

In order to obtain the number of all $(m,n,k)$-Gog trapezoids with $p$ minima and $q$ maxima also if $m \not= 0$, one has to sum over all possible bottom rows and take into account whether the rightmost entry in the bottom row is a maximum. In the expressions in Theorem~\ref{gog}, the bottom row appears only through the factors $\prod\limits_{i=2}^{k} (1+x_i)^{b_i}$ and $\prod\limits_{i=1}^{k} (1+x_i)^{b_i}$, respectively.

\subsection{Gog pentagons} In this section, we consider generalizations of 
Gog trapezoids (so-called \emph{Gog pentagons}) that were recently introduced by Biane and Cheballah \cite{ASMpentagons}. In that paper also certain Magog-type objects were introduced 
(so-called \emph{GOGAm pentagons}) along with the conjecture that there is the same number of Gog pentagons of a given type as there is of GOGAm pentagons of the same type and that this is still true if we restrict to pentagons that have a prescribed bottom entry. 
In this section, we use Corollary~\ref{constant1} to derive constant term expressions for various refined countings of Gog pentagons including fixing the bottom entry.

Let $m,n,k_L,k_R$ be non-negative integers with $n+1 \le k_L + k_R$. An \emph{$(m,n,k_L,k_R)$-Gog pentagon} is an arrangement of positive integers of the following form 
$$
\begin{array}{ccccccccccc}
& & & & & & \bullet & & & &  \\
& & & & & \bullet & & \bullet & & &  \\
& & & & \bullet & & \bullet & & \bullet & &  \\
& & & \bullet & & \bullet & & \bullet & & \bullet &  \\
& & \bullet & & \bullet & & \bullet & & \bullet & & \bullet \\
& \bullet & & \bullet  & & \bullet & & \bullet & & \bullet & \\
 & & \bullet & & \bullet  & & \bullet & & \bullet & &
\end{array},
$$
where $n$ is the number of rows ($n=7$ in the example), $k_L$ is 
the number of NE-diagonals ($k_L=5$ in the example) and $k_R$ is the number 
of SE-diagonals ($k_R=6$ in the example).
Moreover, the entries in the $i$-th NE-diagonal, $1 \le i \le k_L$ (counted from the left), are bounded from below by $i$ and the entries in the $i$-th SE-diagonal, $1 \le i \le k_R$ (counted from the \emph{right}), are bounded from above by $m+n+1-i$. The entries are weakly increasing along NE-diagonals and 
SE-diagonals, and strictly increasing along rows. (By the strict increase along rows, it suffices to require that the entries in the $i$-th NE-diagonal are bounded from below by $i$ for $1 \le i \le n-k_R+1$ and that the entries in the $i$-th SE-diagonal are bounded from above by $m+n+1-i$ for $1 \le i \le n-k_L+1$.) 
A $(3,7,5,6)$-Gog pentagon is displayed next.
$$
\begin{array}{ccccccccccc}
& & & & & & 4 & & & &  \\
& & & & & 3 & & 5 & & &  \\
& & & & 2 & & 5 & & 8 & &  \\
& & & 2 & & 4 & & 7 & & 9 &  \\
& & 1 & & 4 & & 6 & & 7 & & 10 \\
& 1 & & 4 & & 5 & & 8 & & 8 & \\
 & & 3 & & 5 & & 6 & & 8 & &
\end{array}
$$
Obviously, $(m,n,k)$-Gog trapezoids are $(m,n,k,n)$-Gog pentagons. Gog pentagons were first defined by Biane and Cheballah in \cite{ASMpentagons}; they use different parameters: $(m,n,k_L,k_R)$-Gog pentagons as defined here correspond to their $(m+n,k_R,k_L,n)$-Gog pentagons. Also, as we reflected their pentagons along a horizontal axis, the bottom entry in their pentagons is the top entry of the pentagons as defined here.

\subsection{Inversion numbers}
We observe that $(m,n,k_L,k_R)$-Gog pentagons with bottom row 
$b_{n-k_R+1},b_{n-k_R+2},\ldots,b_{k_L}$ correspond to $(\mathbf{s},\mathbf{t})$-trees of order $n$, where $\mathbf{s}=(n-k_R-1,n-k_R-2,\ldots,0)$ and $\mathbf{t}=(0,1,\ldots,n-k_L-1)$. This implies that the number of $(m,n,k_L,k_R)$-Gog pentagons with bottom row $b_{n-k_R+1},\ldots,b_{k_L}$ is equal to 
\begin{multline*}
(-1)^{\binom{n-k_R}{2}}
\prod_{i=1}^{n-k_R} (1+x_i)^i x_i^{1-k_R-i} \prod_{i=n-k_R+1}^{k_L} (1+x_i)^{b_i} x_i^{-n+1} \prod_{i=k_L+1}^{n} (1+x_i)^{m+k_L+1} x_i^{-n+i-k_L} \\
\times \prod_{1 \le i < j \le n} (x_i-x_j)(1+x_j+x_i x_j).
\end{multline*}
If we define the inversion numbers for Gog pentagons in the obvious way, it follows that the generating function 
of $(m,n,k_L,k_R)$-Gog pentagons with bottom row 
$b_{n-k_R+1},\ldots,b_{k_L}$ and w.r.t. the two inversion numbers are 
\begin{multline*}
(-1)^{\binom{n-k_R}{2}}
\prod_{i=1}^{n-k_R} (1+x_i)^i (1+(1-v)x_i)^{-n+k_R+i-1} x_i^{1-k_R-i} \prod_{i=n-k_R+1}^{k_L} (1+x_i)^{b_i} x_i^{-n+1} \\
\times \prod_{i=k_L+1}^{n} (1+x_i)^{m+i+1} (1+u x_i)^{k_L-i} x_i^{-n+i-k_L} 
 \prod_{1 \le i < j \le n} (x_i-x_j)(1+ (1-v)x_i+u(x_j+x_i x_j)).
\end{multline*}

In order to obtain the generating function of $(m,n,k_L,k_R)$-Gog pentagons with top row $a$, one has to multiply the expression with $P_{n,a,\min,\max}(\mathbf{x})$, where $\min \le a \le \max$ and $\min \le 1$ and $\max \ge m+n$. 

\subsection{Minima and maxima}
A \emph{bottom-minimum} is an entry in the leftmost SE-diagonal that is equal to the lower bound of its SE-diagonal, while a \emph{bottom-maximum} is an entry in the rightmost NE-diagonal that is equal to the upper bound of its NE-diagonal. A \emph{top-minimum} is an entry equal to $1$ and such entries can only be in the leftmost NE-diagonal. A \emph{top-maximum} is an entry equal to $m+n$ and such entries can only be in the rightmost SE-diagonal. Observe that minima (resp. maxima) as defined for Gog trapezoids are top-minima (resp. bottom-maxima) as defined 
for Gog pentagons.

Also in this case, it is not possible to include both, the inversions numbers, and the numbers of maxima and minima because for the numbers of minima and maxima we need to apply Corollary~\ref{constant1} in instances where $b_1,\ldots,b_n$ is not necessarily strictly increasing.

Corollary~\ref{constant1} implies in a similar way as for Gog trapezoids that the generating function of $(m,n,k_L,k_R)$-Gog pentagons with bottom row  $b_{n-k_R+1},\ldots,b_{k_L}$ 
w.r.t. the weight
$$
Q_L^{\# \text{of bottom minima not including $b_{n-k_R+1}$}}
Q_R^{\# \text{of bottom maxima not including $b_{k_L}$}}
$$
is the constant term of
\begin{multline*}
(-1)^{\binom{n-k_R}{2}}
\prod_{i=1}^{n-k_R} \left(1- Q_L \frac{x_i}{1+x_i} \right) (1+x_i)^{i+1} x_i^{1-k_R-i} \prod_{i=n-k_R+1}^{k_L} (1+x_i)^{b_i} x^{-n+1} \\
\times \prod_{i=k_L+1}^{n}(1+Q_R x_i) (1+x_i)^{m+k_L} x_i^{-n+i-k_L}  \prod_{1 \le i < j \le n} (x_i-x_j)(1+x_j+x_i x_j)
\end{multline*}
w.r.t.\ $x_1,\ldots,x_n$.
Finally, the generating function of $(m,n,k_L,k_R)$-Gog pentagons with bottom row $b_{n-k_R+1},\ldots,b_{k_L}$ that have at least one top-minimum and at least one top-maximum w.r.t. the weight
$$
P_L^{\# \text{of top minima}} P_R^{\# \text{of top maxima}}
Q_L^{\# \text{of bottom minima not including $b_{n-k_R+1}$}}
 Q_L^{\# \text{of bottom maxima not including $b_{k_L}$}}
$$
is the constant term of the following expression in $x_1,\ldots,x_n$.
\begin{multline*}
 (-1)^{\binom{n-k_R}{2}+1} P_L (1+P_L x_1)^{-1} (1+x_1)^2 x_1^{-k_R+1}
P_R (1+x_n - P_R x_n)^{-1} x_n^{-k_L+1} (1+x_n)^{k_L+m} \\
\times
Q_L Q_R \prod_{i=2}^{n-k_R} \left(1- Q_L \frac{x_i}{1+x_i} \right)(1+x_i)^{i+1} x_i^{1-k_R-i} \prod_{i=n-k_R}^{k_L} (1+x_i)^{b_i} x_i^{-n+1} \\
\times \prod_{i=k_L+1}^{n-1} (1+Q_R x_i) (1+x_i)^{m+k_L} x_i^{-n+i-k_L} 
 \prod_{1 \le i < j \le n} (x_i-x_j)(1+ x_j+x_i x_j).
\end{multline*}
In order to obtain the generating functions of Gog pentagons where the top row is fixed, one has to  proceed as in the previous subsection.

\section{Constant term formula for the number of $(m,n,k)$-Magog trapezoids with prescribed numbers of minima and maxima} 
\label{magog}

In this section we derive constant term formulas for the number of $(m,n,k)$-Magog trapezoids with $p$ maxima and $q$ minima. Together with Theorem~\ref{gog}, where a constant term formula for the number of $(m,n,k)$-Gog trapezoids with $p$ minima and $q$ maxima is provided, this could serve as a framework to give a computational proof of Conjecture~\ref{conj_christ}.

In special cases such constant term formulas have been derived before: Zeilberger gave a constant term expression for the unrestricted enumeration of $(0,n,k)$-Magog trapezoids in
\cite{ZeilbergerASMProof} (see also \cite{zeilbergerconstant}), Krattenthaler extended this to $(m,n,k)$-Magog trapezoids in \cite{KrattGogMagog}, and Ishikawa finally included the number of maxima \cite{Ishikawa}. Their formulas are different from ours as the rational functions underlying the constant term formulas involve in general determinants.

\subsection{First version}
Using a standard technique, we can transform $(m,n,k)$-Magog trapezoids with prescribed bottom row $(b_{n-k+1},b_{n-k+2},\ldots,b_n)$ into families of non-intersecting lattice paths as follows: We start by adding a new leftmost SE-diagonal consisting entirely of $1$'s and a new rightmost SE-diagonal consisting of $(m+1,m+2,\ldots,m+n)$ as indicated in green in our running Example \ref{magogexample}: 
$$
\begin{array}{ccccccccccccccc}
& & & & & & &  \color{green} 3  &  & &  & &&& \\
& & & & & & \color{blue} 3 & & \color{green} 4  & & &&&& \\
& & & & & 2 & & 3 & & \color{green} 5 & &&&&  \\
& & & & 2 & & 3 & & 4 & & \color{green} 6 &&&& \\
& & & 1 & & 3 & & 4  & & 5 & & \color{green} 7 &&& \\
& & \color{red} 1 & & 2 & & 3 & & 5 & & \color{blue} 7 && \color{green} 8 &&  \\
&  \color{green}  1 & &  \color{red} 1 & & 3 & & 4 & & 6 & & 7 & & \color{green} 9 \\
 & & \color{green} 1 & & 2 & & 3 & & 4 & & 7 & & \color{blue} 9
\end{array}.
$$
Next, we add $i-1$ to the $i$-th NE-diagonal, counted from the left, for all $1 \le i \le n$. Then we associate with each NE-diagonal a lattice path with north steps and east steps where the entries of the NE-diagonals are the heights of the paths, and, as long as $1 \le i \le n-k-1$,  the $x$-coordinate of the starting point is shifted by one unit to the left when passing from the $i$-th NE-diagonal to the $(i+1)$-st NE-diagonal, while these $x$-coordinates are the same for the $k+1$ rightmost NE-diagonals. The family of non-intersecting lattice paths in Figure~\ref{nilp} corresponds to the Magog in \eqref{magogexample}. Note that each of the paths starts and ends with a horizontal step and thus we cut off these horizontal steps in the following.

\begin{figure}
\scalebox{0.8}{
\psfrag{1}{\large $1$}
\psfrag{2}{\large$2$}
\psfrag{3}{\large$3$}
\psfrag{4}{\large$4$}
\psfrag{5}{\large$5$}
\psfrag{6}{\large$6$}
\psfrag{7}{\large $7$}
\psfrag{8}{\large $8$}
\psfrag{9}{\large $9$}
\psfrag{10}{\large $10$}
\psfrag{11}{\large $11$}
\psfrag{12}{\large $12$}
\psfrag{13}{\large $13$}
\psfrag{14}{\large $14$}
\psfrag{15}{\large $15$}
\psfrag{0}{\large $(0,0)$}
\includegraphics{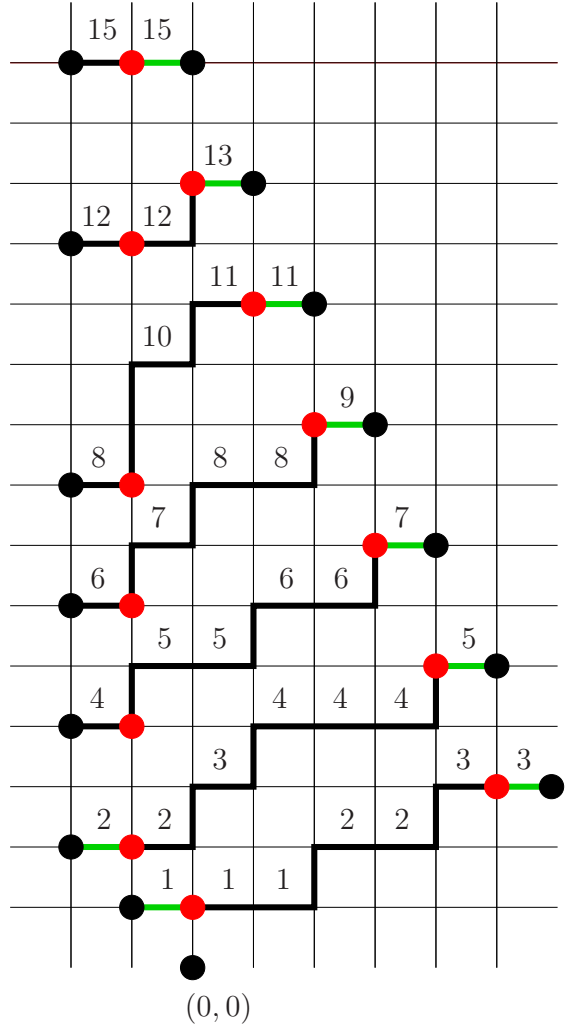}}\caption{\label{nilp} Non-intersecting lattice paths}\end{figure}

Now it can be deduced that $(m,n,k)$-Magogs with bottom row $(b_{n-k+1},b_{n-k+1},\ldots,b_n)$ correspond to the families of non-intersecting lattice paths starting at $(-i+1,i)$, $1 \le i \le n-k$, and $(-n+k+1,b_i+i-1)$, $n-k+1 \le i \le n$, and ending at $(k-j+1,m+2j-1)$, $1 \le j \le n$, allowing steps in north direction and east direction. The last entry of a NE-diagonal is a maximum if  the last step of the respective lattice path is horizontal, while the first entry of the $n-k$ leftmost NE-diagonals is a minimum if the first step of the respective path is horizontal. The following generating function of lattice paths starting at $(a,b)$ and ending at $(c,d)$ 
$$
\sum_{\text{lattice path $(a,b) \to (c,d)$}} 
Q^{[\text{first step is horizontal}]} P^{[\text{last step is horizontal}]}
$$
allowing steps in north direction and east direction is 
\begin{multline*}
\mathcal{N}((a,b+1) \to (c,d-1)) + Q \, \mathcal{N}((a+1,b) \to (c,d-1)) + \\ 
P \, \mathcal{N}((a,b+1) \to (c-1,d)) + P Q \, \mathcal{N}((a+1,b) \to (c-1,d)),
\end{multline*}
where $\mathcal{N}((a,b) \to (c,d))$ denotes the number of lattice paths starting at 
$(a,b)$ and ending at $(c,d)$ (which is equal to $\binom{c+d-a-b}{d-b}$). Thus, 
this generating function is 
$$
\binom{c+d-a-b-2}{d-b-2} + (P+Q) \binom{c+d-a-b-2}{d-b-1} + 
P Q \binom{c+d-a-b-2}{d-b},
$$
which evaluates to 
$$
\binom{c+d-a-b-1}{d-b-1} + P \binom{c+d-a-b-1}{d-b}
$$
if $Q=1$. If we use 
$$
\binom{n}{k} = \begin{cases} \frac{n(n-1) \cdots (n-k+1)}{k!} & k \ge 0 \\ 0 & k < 0 \end{cases},  
$$
then the formulas for the generating functions are true for any choice of integers $a,b,c,d$ (in particular also when there is no path from $(a,b)$ to $(c,d)$) if whenever the upper parameter of each binomial coefficient is negative, then also the lower parameter is negative. 
The Lindstr\"om-Gessel-Viennot theorem \cite{Lindstr,GesselViennot} now implies that the generating function of $(m,n,k)$-Magog trapezoids that have bottom row $(b_{n-k+1},\ldots,b_{n})$ is 
$$
\det_{1 \le i, j \le n} \left( \begin{cases} \binom{j+k+m-3}{2j+m-i-3} + 
(P+Q) \binom{j+k+m-3}{2j+m-i-2} + P Q \binom{j+k+m-3}{2j+m-i-1} , & i=1,\ldots,n-k \\
\binom{j+m+n-b_i-i-1}{2j+m-b_i-i-1} + P \binom{j+m+n-b_i-i-1}{2j+m-b_i-i}, & i=n-k+1,\ldots,n \end{cases} \right),
$$
where the weight is
$$
P^{\# \text{ of maxima}} \cdot Q^{\# \text{ of minima not contained in the bottom row}}.
$$
In case $1 \le i \le n-k$, the upper parameter of a binomial coefficient in the matrix is negative only if $j=1, k=1, m=0$; then all lower parameters are also negative, except the one after $PQ$ if $i=1$ which is $0$ and it can be checked that also in this case we have the correct generating function. In case $n-k+1 \le i \le n$, the upper parameters are either no less than the lower parameters or non-negative or the lower parameters are negative, except for the binomial coefficient after $P$ when 
$j+m+n-b_i-i-1=-1$ and $2j+m-b_i-i=0$, and it can be checked that also in this case we obtain the correct result.

Using $\binom{n}{k} = (-1)^{k} \binom{k-n-1}{k}$, this is equal to
\begin{multline*}
(-1)^{
(m-1)n + \binom{n+1}{2} + \sum\limits_{i=n-k+1}^{n} b_i} \\
\times \det_{1 \le i, j \le n} \left( \begin{cases} \binom{j-i-k-1}{2j+m-i-3} - (P+Q) \binom{j-i-k}{2j+m-i-2} + P Q \binom{j-i-k+1}{2j+m-i-1}, & i=1,\ldots,n-k \\
\binom{j-n-1}{2j+m-b_i-i-1} - P \binom{j-n}{2j+m-b_i-i}, & i=n-k+1,\ldots,n \end{cases} \right).
\end{multline*}
Now, as $\ct_x \frac{(1+x)^n}{x^k} = \binom{n}{k}$, where $\ct_x$ denotes the constant term in $x$, this is equal to 
\begin{multline*}
\ct_{x_1,\ldots,x_n} (-1)^{(m-1)n + \binom{n+1}{2} + \sum\limits_{i=n-k+1}^{n} b_i} \\
\times \det_{1 \le i, j \le n} \left( \begin{cases} \frac{(1+x_i)^{j-i-k-1}}{x_i^{2j+m-i-3}} - (P+Q)  \frac{(1+x_i)^{j-i-k}}{x_i^{2j+m-i-2}} + P Q  
 \frac{(1+x_i)^{j-i-k+1}}{x_i^{2j+m-i-1}},  & i=1,\ldots,n-k \\
 \frac{(1+x_i)^{j-n-1}}{x_i^{2j+m-b_i-i-1}} - P \frac{(1+x_i)^{j-n}}{x_i^{2j+m-b_i-i}}, & i=n-k+1,\ldots,n \end{cases} \right).
\end{multline*}
After pulling out the factor 
$$
\prod_{i=1}^{n-k} \frac{(1+x_i)^{-i-k}}{x_i^{m-i+1}} 
\left( x_i -P (1+x_i) \right) 
\left( x_i -Q (1+x_i) \right)  
\prod_{i=n-k+1}^{n} \frac{(1+x_i)^{-n}}{x_i^{m-b_i-i+2}} 
\left( x_i -P (1+x_i) \right)
$$
of the determinant, it remains 
$$
\det_{1 \le i, j \le n} \left( (x_i^{-1} + x_i^{-2})^{j-1} \right) = 
\prod_{1 \le i < j \le n} \left( x_j^{-1}+x_j^{-2}-x_i^{-1}-x_i^{-2} \right) =
\prod_{1 \le i < j \le n} \frac{(x_i-x_j)(x_i+x_j+x_i x_j)}{x_i^2 x_j^2},
$$
which can be computed using the Vandermonde determinant evaluation. 

We obtain the following theorem.

\begin{theo}
The generating function of $(m,n,k)$-Magog trapezoids with bottom row 
$b_{n-k+1},\ldots,b_n$ and w.r.t. the weight
$$
P^{\# \text{of maxima}} Q^{\# \text{of minima not including $b_{n-k+1}$}}
$$
is the constant term of the following expression in $x_1,\ldots,x_n$.
\begin{multline}
\label{magog1}
(-1)^{(m-1)n + \binom{n+1}{2} + \sum\limits_{i=n-k+1}^{n} b_i}
\prod_{i=1}^{n-k} (1+x_i)^{-i-k} x_i^{-m-2n+i+1} 
\left( x_i -P (1+x_i) \right) 
\left( x_i -Q (1+x_i) \right)  \\ \times
\prod_{i=n-k+1}^{n} (1+x_i)^{-n} x_i^{-m-2n+b_i+i} 
\left( x_i -P (1+x_i) \right)
\prod_{1 \le i < j \le n} (x_i-x_j)(x_i+x_j+x_i x_j).
\end{multline}
\end{theo}

We set 
\begin{multline*}
M_{m,n,k}(x_1,\ldots,x_n)  \\
= (-1)^{(m-1)n + \binom{n-k+1}{2}} 
\prod_{i=1}^{n-k} (1+x_i)^{-i-k} x_i^{-m-2n+i+1} 
\left( x_i -P (1+x_i) \right) 
\left( x_i -Q (1+x_i) \right)  \\ \times
\prod_{i=n-k+1}^{n} (1+x_i)^{-n} x_i^{-m-2n} 
\left( x_i -P (1+x_i) \right)
\prod_{1 \le i < j \le n} (x_i-x_j)(x_i+x_j+x_i x_j), 
\end{multline*}
so that the generating function is 
$$
\prod_{i=n-k+1}^{n} (-x_i)^{b_i+i} M_{m,n,k}(x_1,\ldots,x_n).
$$
It will be crucial that $M_{m,n,k}(x_1,\ldots,x_n)$ is an antisymmetric function 
in $x_{n-k+1},\ldots,x_n$. 
It follows that the generating function of $(m,n,k)$-Magogs w.r.t. the weight
$$
P^{\# \text{ of maxima}} \cdot Q^{\# \text{ of minima}}
$$
is the constant term of 
\begin{multline}
\label{magog2}
\left( Q \left( \sum_{1 \le b_{n-k+2} \le b_{n-k+3} \le \ldots \le b_n} y_{n-k+1}^{n-k+2} 
\prod_{i=n-k+2}^{n} y_{i}^{b_i+i} \right) + 
\left( \sum_{2 \le b_{n-k+1} \le b_{n-k+2} \le \ldots \le b_n}  
\prod_{i=n-k+1}^{n} y_{i}^{b_i+i} \right) \right) \\
\times M_{m,n,k}(x_1,\ldots,x_n) \\
= \left( Q + \frac{\prod_{i=n-k+1}^{n} y_i}{1-\prod_{i=n-k+1}^{n} y_i} \right)
\frac{\prod_{i=n-k+1}^{n} y_i^{i+1}}{\prod_{i=n-k+2}^{n} \left(1-\prod_{j=i}^{n} y_j \right)} M_{m,n,k}(x_1,\ldots,x_n),
\end{multline}
where we set $y_i=-x_i$.
We use the following notation: Suppose $F(x_{r+1},\ldots,x_{r+s})$ is a function in $x_{r+1},\ldots,x_{r+s}$ and $\sigma \in {\mathcal S}_s$, then 
$$
\sigma [ F(x_{r+1},\ldots, x_{r+s}) ]:= F(x_{\sigma(r+1)},\ldots,x_{\sigma(r+s)}).
$$
Now the constant term of the expression in \eqref{magog2} is the constant term of
\begin{multline}
\label{magog3}
\frac{1}{(k-1)!} \sum_{\sigma \in {\mathcal S}_{k-1}} Q \, y_{n-k+1}^{n-k+2} 
M_{m,n,k}(x_1,\ldots,x_n) 
\sgn \sigma \cdot \sigma \left[ \prod_{i=n-k+2}^{n} \frac{y_i^{i+1}}{1-\prod_{j=i}^{n} y_j} \right]\\
 + \frac{1}{k!} \sum_{\sigma \in {\mathcal S}_{k}} M_{m,n,k}(x_1,\ldots,x_n) \sgn \sigma \cdot \sigma \left[ \prod_{i=n-k+1}^{n} \frac{y_i^{i+2}}{1-\prod_{j=i}^{n} y_j} \right].
\end{multline}

We need the following lemma, which can be found in 
\cite[Example 4, Chapter III.5]{Macdonald}. It appeared in a similar context in  \cite[Subsublemma 1.1.3]{ZeilbergerASMProof}. 

\begin{lem} 
\label{zeilberger}
Let $r \ge 1$ be an integer. Then 
$$
\sum_{\sigma \in  {{\mathcal S}_r}} \sgn \sigma \cdot \sigma \left[ \prod\limits_{i=1}^{r} 
\frac{y_i^{i-1}}{1 - \prod_{j=i}^{r} y_j} \right] = \prod\limits_{i=1}^{r} (1-y_i)^{-1} 
\prod_{1 \le i < j \le r} \frac{y_j-y_i}{1-y_i y_j}.
$$
\end{lem}

From the lemma it now follows that the expression in \eqref{magog3} is equal to 
\begin{multline*}
\frac{Q \cdot y_{n-k+1}^{n-k+2} M_{m,n,k}(x_1,\ldots,x_n) \prod_{i=n-k+2}^{n} y_i^{n-k+3} (1-y_i)^{-1} }{(k-1)!}  \prod_{n-k+2 \le i < j \le n} \frac{y_j-y_i}{1-y_i y_j} \\ 
+ \frac{ M_{m,n,k}(x_1,\ldots,x_n) \prod_{i=n-k+1}^{n} y_i^{n-k+3} (1-y_i)^{-1} }{k!} \prod_{n-k+1 \le i < j \le n} \frac{y_j-y_i}{1-y_i y_j}.
\end{multline*}
This can also be written as
\begin{multline*}
\frac{Q \cdot y_{n-k+1}^{n-k+2} M_{m,n,k}(x_1,\ldots,x_n) \prod_{i=n-k+2}^{n} (1-y_i)^{-1}}{(k-1)!}  \prod_{n-k+2 \le i < j \le n} \frac{1}{1-y_i y_j} \sum_{\sigma \in {\mathcal S}_{k-1}} \sgn \sigma \, 
\sigma \left[ \prod_{i=n-k+2}^{n} y_i^{i+1} \right] \\ 
+ \frac{ M_{m,n,k}(x_1,\ldots,x_n) \prod_{i=n-k+1}^{n} (1-y_i)^{-1} }{k!} \prod_{n-k+1 \le i < j \le n} \frac{1}{1-y_i y_j} \sum_{\sigma \in {\mathcal S}_k} \sgn \sigma \, 
\sigma \left[ \prod_{i=n-k+1}^{n} y_i^{i+2} \right].
\end{multline*}
We again employ the antisymmetry of $M_{m,n,k}(x_1,\ldots,x_n)$ to see that the constant term of the previous expression this is equal to the constant term of
\begin{multline*}
Q \cdot y_{n-k+1}^{n-k+2} M_{m,n,k}(x_1,\ldots,x_n) \prod\limits_{i=n-k+2}^{n} y_{i}^{i+1} (1-y_i)^{-1}  \prod_{n-k+2 \le i < j \le n} \frac{1}{1-y_i y_j} \\ 
+ M_{m,n,k}(x_1,\ldots,x_n) \prod\limits_{i=n-k+1}^{n} y_i^{i+2} (1-y_i)^{-1} \prod\limits_{n-k+1 \le i < j \le n} \frac{1}{1-y_i y_j} \\
= (-1)^{\binom{n+3}{2} +\binom{n-k+3}{2}} M_{m,n,k}(x_1,\ldots,x_n) \prod_{i=n-k+1}^{n} \frac{x_i^{i+2}}{1+x_i} \prod_{n-k+1 \le i < j \le n} \frac{1}{1-x_i x_j} \\
\times \left((-1)^{k} Q (1+x_{n-k+1}^{-1}) \prod_{j=n-k+2}^{n} (x_j^{-1}-x_{n-k+1}) +1 \right).
\end{multline*}

We summarize our result in the following theorem. 

\begin{theo}
\label{ver1} 
The generating function of $(m,n,k)$-Magog trapezoids w.r.t. the weight
$$
P^{\# \text{of maxima}} \, Q^{\# \text{of minima}}
$$
is the constant term in $x_1,\ldots,x_n$ of the following expression. 
\begin{multline*}
(-1)^{(m-1)n +\binom{n+1}{2}}
\left((-1)^k Q (1+x_{n-k+1}^{-1}) \prod_{j=n-k+2}^{n} (x_j^{-1}-x_{n-k+1}) +1 \right)
\\ \times 
\prod_{i=1}^{n-k} (1+x_i)^{-i-k} x_i^{-m-2n+i+1} 
\left( x_i -P (1+x_i) \right) 
\left( x_i -Q (1+x_i) \right)  
\\
\times
\prod_{i=n-k+1}^{n} (1+x_i)^{-n-1} x_i^{-m-2n+i+2} 
\left( x_i -P (1+x_i) \right)
\\
\times
\prod_{1 \le i < j \le n} (x_i-x_j)(x_i+x_j+x_i x_j) \prod_{n-k+1 \le i < j \le n} \frac{1}{1-x_i x_j}
\end{multline*}
\end{theo}

\subsection{Second version}

There are other options to encode $(m,n,k)$-Magog trapezoids as families of non-intersecting lattice paths, see for instance \cite{KrattGogMagog}. We derive the constant term expression for a second possibility in this subsection, following an encoding that was used for instance in \cite{FonZinn}. Instead of interpreting NE-diagonals as lattice paths, we now consider the lattice paths that separate the entries that are less than or equal to $i$ in the Magog trapezoid from the entries that are greater than or equal to $i+1$, $1 \le i \le m+n-1$, see Figure~\ref{magog_ex}. By rotating the picture and shifting the paths appropriately, these paths can be transformed into a family of non-intersecting lattice paths, see 
Figure~\ref{nilp2} (the leftmost separating path in Figure~\ref{magog_ex}, i.e.\ the one separating the $1$'s from $2$'s, corresponds to the bottom path in Figure~\ref{nilp2}).

\begin{figure}
\scalebox{0.6}{
\psfrag{1}{\large $1$}
\psfrag{2}{\large$2$}
\psfrag{3}{\large$3$}
\psfrag{4}{\large$4$}
\psfrag{5}{\large$5$}
\psfrag{6}{\large$6$}
\psfrag{7}{\large $7$}
\psfrag{8}{\large $8$}
\psfrag{9}{\large $9$}
\psfrag{10}{\large $10$}
\psfrag{11}{\large $11$}
\psfrag{12}{\large $12$}
\psfrag{13}{\large $13$}
\psfrag{14}{\large $14$}
\psfrag{15}{\large $15$}
\psfrag{0}{\large $(0,0)$}
\includegraphics{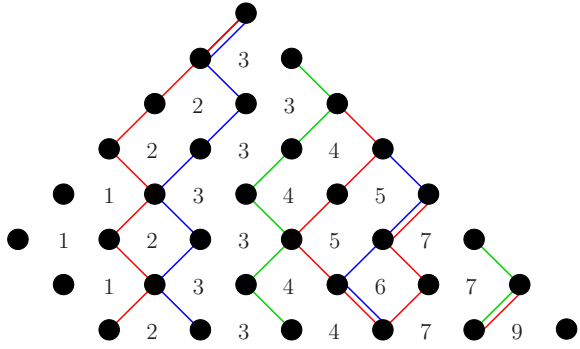}}\caption{\label{magog_ex} An $(2,7,5)$-Magog trapezoids and its separating lattice paths}\end{figure}

\begin{figure}
\scalebox{0.8}{
\psfrag{1}{\large $1$}
\psfrag{2}{\large$2$}
\psfrag{3}{\large$3$}
\psfrag{4}{\large$4$}
\psfrag{5}{\large$5$}
\psfrag{6}{\large$6$}
\psfrag{7}{\large $7$}
\psfrag{8}{\large $8$}
\psfrag{9}{\large $9$}
\psfrag{10}{\large $10$}
\psfrag{11}{\large $11$}
\psfrag{12}{\large $12$}
\psfrag{13}{\large $13$}
\psfrag{14}{\large $14$}
\psfrag{15}{\large $15$}
\psfrag{0}{\large $(0,0)$}
\includegraphics{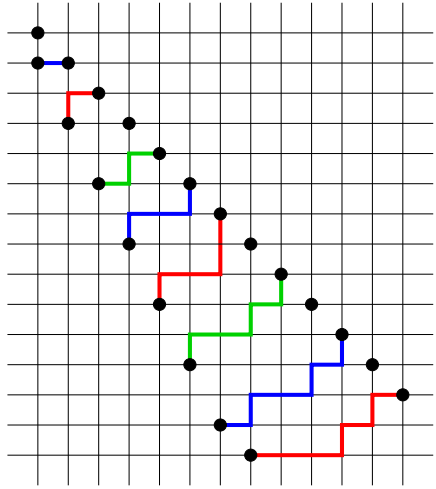}}\caption{\label{nilp2} Non-intersecting lattice paths for the second version}\end{figure}

It can be worked out that $(m,n,k)$-Magog trapezoids correspond to families of non-intersecting lattice paths with starting points  $(-i+1,i-1)$, $i=1,\ldots,m-1$, and $(-i+1,2i-m-1)$, $m \le i \le m+n-1$, and end points 
$(k-b_j+1,n-k+b_j-1)$, $1 \le j \le m+n-1$, for some 
$1 \le b_1 < b_2 < \ldots < b_{m+n-1} \le m+n+k-1$, 
with north steps and east steps. The maximums correspond to the lattice paths with starting points $(-i+1,2i-m-1)$, $m \le i \le m+n-1$, such that the first step is horizontal, while the number of minima is just the height of the last horizontal step of the path ending at $(k-b_1+1,n-k+b_1-1)$ if $b_1=1$, otherwise it is $n-k+1$. 
It follows that the number of
$(m,n,k)$-Magog trapezoids with $q$ minima is equal to the number of the 
following family of lattice paths depending on whether or not $q=n-k+1$.

\smallskip

\emph{Case~$q \le n-k$:}
In this case, we need to count families of non-intersecting lattice paths starting at $(-i+1,i-1)$, $i=1,\ldots,m-1$, 
and $(-i+1,2i-m-1)$, $m \le i \le m+n-1$, and with ending points 
$(k-1,q)$ and $(k-b_j+1,n-k+b_j-1)$, $2 \le j \le m+n-1$, for some 
$2 \le  b_2 < \ldots < b_{m+m-1} \le m+n+k-1$. 

\smallskip

\emph{Case~$q=n-k+1$:} 
Here, we need to count families of non-intersecting lattice paths starting at $(-i+1,i-1)$, $i=1,\ldots,m-1$, 
and $(-i+1,2i-m-1)$, $m \le i \le m+n-1$, and with ending points $(k-b_j+1,n-k+b_j-1)$, $1 \le j \le m+n-1$, for some 
$2 \le  b_1 < \ldots < b_{m+m-1} \le m+n+k-1$. 

\smallskip

First we assume $q \le n-k$.

The generating function of lattice paths starting at $(a,b)$ and ending at 
$(c,d)$ with respect to the weight 
$$P^{[\text{first step is horizontal}]}$$
allowing steps in north direction and east direction is 
\begin{multline*}
P \cdot {\mathcal N}((a+1,b) \to (c,d)) + {\mathcal N}((a,b+1) \to (c,d)) \\
= P \binom{c+d-a-b-1}{c-a-1} + \binom{c+d-a-b-1}{c-a}.
\end{multline*}
The formula is true whenever $c+d-a-b > 0$. It is also true when $c+d-a-b=0$ and $c-a=0$. If $c+d-a-b=0$ and $c -a \not=0$ or 
$c+d-a-b < 0$, then the generating function vanishes.
It follows that the generating function of $(m,n,k)$-Magog trapezoids 
with $q$ minima and w.r.t. the number of maxima is 
$$ 
\det_{1 \le i, j \le m+n-1}  
\left( \begin{array}{c|c}
\binom{k+q-1}{k+i-2}  &  \binom{n}{i+k-b_j}  \\ \hline 
\begin{cases} P \binom{k+m+q-i-2}{k+i-3} + \binom{k+m+q-i-2}{k+i-2},  & i \le k+m+q-2 \\ 
0, & i \ge k+m+q-1 \end{cases}  & P \binom{m+n-i-1}{i+k-b_j-1} + \binom{m+n-i-1}{i+k-b_j}   
\end{array} \right), 
$$
where we need to distinguish between $1 \le i \le m-1$ and $m \le i \le m+n-1$ concerning the row, 
and between $j=1$ and $2 \le j \le m+n-1$ concerning the columns, and we have to exclude the case when $k=1$ and $m+q=1$.
This is equal to the constant term of 
$$ 
\det_{1 \le i, j \le m+n-1}  
\left( \begin{array}{c|c}
\frac{(1+x_i)^{k+q-1}}{x_i^{i+k-3}}  &  \frac{(1+x_i)^{n}}{x_i^{i+k-b_j}}  \\ \hline 
\begin{cases} P \frac{(1+x_i)^{k+m+q-i-2}}{x_i^{i+k-3}} + \frac{(1+x_i)^{k+m+q-i-2}}{x_i^{i+k-2}}   & i \le k+m+q-2 \\
0 & i \ge k+m+q-1 \end{cases} & P \frac{(1+x_i)^{m+n-i-1}}{x_i^{i+k-b_j-1}} + \frac{(1+x_i)^{m+n-i-1}}{x_i^{i+k-b_j}}   
\end{array} \right).
$$
We pull out the factor 
$$
P^{[m=0]} \prod_{i=1}^{m+n-1} (1+x_i)^n x_i^{-i-k} \prod_{i=\max(1,m)}^{m+n-1} (1+x_i)^{m-i-1} \left(P x_{i} +1\right),
$$
and we obtain the following determinant
$$
\det_{1 \le i, j \le m+n-1}  
\left( \begin{array}{c|c}
\begin{cases}  (1+x_i)^{k+q-n-1} x_{i}^2, & i \le k+m+q-2 \\ 0, & i \ge k+m+q-1 \end{cases}   & x_i^{b_j}  
\end{array} \right),
$$
where we now only have to distinguish between $j=1$ and $2 \le j \le m+n-1$ concerning the columns, while we have a homogeneous definition in the rows. We expand w.r.t.\ the first column and obtain 
\begin{equation}
\label{partial}
\sum_{l=1}^{k+m+q-2} (-1)^{l+1} (1+x_l)^{k+q-n-1} x_l^2 \det_{2 \le j \le m+n-1 \atop 1 \le i \le m+n-1, i\not=l} \left( x_i^{b_j} \right).
\end{equation}
Now we aim to sum over all $2 \le b_2 < \ldots < b_{m+n-1} \le m+n+k-1$. However, since the constant term is zero if $b_{m+n-1} \ge m+n+k$, we can simply sum over all $2 \le b_2 < \ldots < b_{m+n-1}$. Observe that by Lemma~\ref{zeilberger} we get
\begin{multline}
\label{zeilberger2}
\sum_{b \le b_1 < b_2 < \ldots < b_r} \det_{1 \le i,j \le r} 
\left( x_i^{b_j} \right) = \sum_{b \le b_1 < b_2 < \ldots < b_r} 
\sum_{\sigma \in {\mathcal S}_r} \sgn \sigma \cdot \sigma 
\left[ x_1^{b_1} \cdots x_r^{b_r} \right] \\
= \prod_{i=1}^{l} x_i^{b}
\sum_{\sigma \in {\mathcal S}_r} \sgn \sigma \cdot \sigma 
\left[ \sum_{0 \ge b_1 < b_2 < \ldots < b_r} x_1^{b_1} \cdots x_r^{b_r} \right] =
\prod_{i=1}^{r} x_i^{b} \sum_{\sigma \in {\mathcal S}_r} \sgn \sigma 
\cdot \sigma \left[ \prod_{i=1}^{r} \frac{x_i^{i-1}}{ 
\left( 1 - \prod_{j=i}^{r} x_j \right)} \right] \\
= \prod_{i=1}^{r} x_i^{b} (1-x_i)^{-1} \prod_{1 \le i < j \le r} 
\frac{x_j-x_i}{1-x_i x_j}.
\end{multline} 
We sum \eqref{partial} over all $b_2,\ldots,b_{m+n-1}$ with 
$2 \le b_2 \ldots < b_{m+n-1}$ and obtain
$$
\sum_{l=1}^{k+m+q-2} (-1)^{l+1} (1+x_l)^{k+q-n-1} x_l^2 
\prod_{1 \le i \le m+n-1 \atop i \not= l}  x_i^{2} (1-x_i)^{-1}  \prod_{1 \le i < j \le m+n-1 \atop i,j \not= l} 
\frac{x_j-x_i}{1-x_i x_j}
$$
It follows that the generating function of $(m,n,k)$-Magog trapezoids with 
$q$ Minima w.r.t.\ the weight $P^{\# \text{ of maxima}}$ is the constant term of
\begin{multline*}
P^{[m=0]} \prod_{i=1}^{m+n-1} (1+x_i)^n x_i^{-i-k} \prod_{i=\max(m,1)}^{m+n-1} (1+x_i)^{m-i-1} \left(P x_{i}+1\right) \\
\times \sum_{l=1}^{k+m+q-2} (-1)^{l+1} (1+x_l)^{k+q-n-1} x_l^{2}
\prod_{1 \le i \le m+n-1 \atop i \not= l}^{m+n-1} x_i^{2} (1-x_i)^{-1}  \prod_{1 \le i < j \le m+n-1 \atop i,j \not= l} 
\frac{x_j-x_i}{1-x_i x_j}
\end{multline*}
if we assume $q \le n-k$.  

In case $q=n-k+1$, the number of Magog trapezoids is the constant term of 
$$ 
\det_{1 \le i, j \le m+n-1}  
\left( 
\begin{cases}
\frac{(1+x_i)^{n}}{x_i^{i+k-b_j}},  & i=1,\ldots,m-1 \\
P \frac{(1+x_i)^{m+n-i-1}}{x_i^{i+k-b_j-1}} + \frac{(1+x_i)^{m+n-i-1}}{x_i^{i+k-b_j}},   & i=m,\ldots,m+n-1
\end{cases}
\right).
$$
Again we pull out the factor 
$$
P^{[m=0]} \prod_{i=1}^{m+n-1}(1+x_i)^n x_i^{-i-k} \prod_{i=\max(m,1)}^{m+n-1} (1+x_i)^{m-i-1} \left(P x_{i} +1\right),
$$
and, in this case, we obtain the following simple determinant.
$$
\det_{1 \le i, j \le m+n-1}  
\left( x_i^{b_j} \right),
$$
Finally, we sum over all $2 \le b_1 < b_2 < \ldots < b_{m+n-1}$, and, by \eqref{zeilberger2}, we obtain 
$$
P^{[m=0]} \prod_{i=1}^{m+n-1} (1+x_i)^n x_i^{-i-k+2}(1-x_i)^{-1} \prod_{i=\max(m,1)}^{m+n-1} (1+x_i)^{m-i-1} \left(P x_{i} +1\right)
\prod_{1 \le i < j \le m+n-1} 
\frac{x_j-x_i}{1-x_i x_j}.
$$

We summarize our results in the following theorem.

\begin{theo}
\label{ver2} 
Suppose $k \not=1$ or $m+q \not=1$.
The generating function of $(m,n,k)$-Magog trapezoids w.r.t. the weight 
$P^{\# \text{ of maxima}}$ and with $q$ minima is equal to the  constant term of 
\begin{multline*}
 P^{[m=0]} \prod_{i=1}^{m+n-1} (1+x_i)^n x_i^{-i-k+2} (1-x_i)^{-1} \prod_{i=\max(m,1)}^{m+n-1} (1+x_i)^{m-i-1} \left(P x_{i} +1\right) \prod\limits_{1 \le i < j \le m+n-1} 
\frac{x_j-x_i}{1-x_i x_j}  \\
\times 
\begin{cases} 
\sum\limits_{i=1}^{k+m+q-2}  (1+x_i)^{k+q-n-1} (1 -x_i) 
\prod\limits_{1 \le j \le m+n-1 \atop j \not=i} 
\frac{1-x_i x_j}{x_j-x_i}, & \text{if $q \le n-k$} \\ 
\qquad \qquad 1, & \text{if $q=n-k+1$.} 
\end{cases}
\end{multline*}
\end{theo}

\section{Connecting Gogs and Magogs}
\label{connect}

In summary, we have seen that the constant term expressions for the number of \emph{Gog-type objects} are all of the following form
$$
s \cdot \prod_{i=1}^{n} f_i(x_i) \prod_{1 \le i < j \le n} (x_i-x_j)(1+x_j+x_i x_j),
$$
where $f_i(x)$ are certain (simple) rational function---often of the form 
$$f_i(x)=x^{a_i}(1+x)^{b_i},$$ 
$a_i,b_i$ integers---and $s$ is a sign. These functions usually have (somewhat) homogeneous definitions for certain ranges of $i$: for instance, when considering $(m,n,k)$-Gog trapezoids, the definition is homogeneous for $1 \le i \le k$ and also for $k+1 \le i \le n$, see Theorem~\ref{gogfirst}. In this case, a factor $1+x$ in $f_i(x)$ is replaced by $1+Q x_i$ for $k+1 \le i \le n$ when considering the number of maxima, and $f_1(x)$ has an exceptional definition when considering the number of minima.
When considering the generating function w.r.t.\ the two inversion numbers, the term
$1+x_j+x_i x_j$ is replaced by $1+(1-v) x_i + u (x_j + x_i x_j)$ and some $f_i(x)$ have an additional factor $\frac{1}{(1+(1-v)x)^{c_i}}$, while others have an additional factor $\frac{1}{(1+u x)^{d_i}}$, $c_i,d_i$ non-negative integers.

Concerning constant term expressions for the number of 
\emph{Magog-type objects} we need to distinguish between the two versions. In the first case (Theorem~\ref{ver1}), we have a sum of two expressions of the form
$$
s \cdot \prod_{i=1}^{n} f_i(x_i) \prod_{1 \le i < j \le n} (x_i-x_j)(x_i+x_j+x_i x_j)
\prod_{n-k+1 \le i < j \le n} \frac{1}{1-x_i x_j},
$$ 
while in the second case (Theorem~\ref{ver2}), we have an expression of the form 
$$
s \cdot \prod_{i=1}^{n} f_i(x_i) \prod_{1 \le i < j \le m+n-1} \frac{x_j-x_i}{1-x_i x_j}
$$
if $q=n-k+1$ and a slightly more complicated expression otherwise.

In \cite{FonZinn} it was shown that the number of $n \times n$ ASMs (which is the same as the number of $(n,n,n)$-Gog trapezoids) is the constant term of the following expression
\begin{equation}
\label{gog-zinnjustin}
\prod_{i=1}^{n-1} (1+x_i)^2 x_i^{-2i+1} \prod_{1 \le i < j \le n-1} (x_j-x_i)(1+x_j+x_i x_j).
\end{equation}
This is a Gog-type constant term expression as described above. Interestingly, this formula was derived using the six-vertex model approach, which is different from the approach that was used here. On the other hand, it was shown that the number of TSSCPPs in an $2n \times 2n \times 2n$ box (which is the same as the number of $(n,n,n)$-Magogs) is the constant term of 
\begin{equation}
\label{magog-zinnjustin}
\prod_{i=1}^{n} (1+x_i)^i x_i^{-2i+2} (1-x_i^2)^{-1} \prod_{1 \le i < j \le n}
 \frac{x_j-x_i}{1-x_i x_j}, 
 \end{equation}
and this is a Magog-type constant term expression. In \cite{FonZinn}, Fonseca and Zinn-Justin succeeded in showing that the constant terms of the two expressions are equal. It can essentially be deduced from the following 
theorem, which was conjectured by Di Francesco and Zinn-Justin \cite{qkz} and first proven by Zeilberger \cite{ZeilDiFrancesco}.
\begin{theo}
\label{zeil}
Let $S(x_1,\ldots,x_n)$ be a power series in $x_1,\ldots,x_n$ that is symmetric in $x_1,\ldots,x_n$. Then 
\begin{multline*}
\ct_{x_1,\ldots,x_n} S(x_1,\ldots,x_n) \prod_{i=1}^{n} x_i^{-2i+1} \prod_{1 \le i < j \le n} (x_j-x_i)(1+t x_j + x_i x_j) \\
= \ct_{x_1,\ldots,x_n} S(x_1,\ldots,x_n) \prod_{i=1}^{n} 
\frac{(1+t x_i)^{i-1} x_i^{-2i+1}}{1-x_i^2} \prod_{1 \le i < j \le n} 
\frac{x_j-x_i}{1-x_i x_j}.
\end{multline*}
\end{theo}
Fonseca and Zinn-Justin \cite{FonZinn} gave another proof of the fact that the constant term of \eqref{gog-zinnjustin} is equal to the constant term of \eqref{magog-zinnjustin}, which is discussed next.

The general strategy (which also appeared in \cite{ZeilbergerASMProof}) is to compute \emph{the symmetrizers} of \eqref{gog-zinnjustin} and of \eqref{magog-zinnjustin}. (The symmetrizer is defined as the antisymmetrizer except for the sign $\sgn \sigma$ which is omitted, see \eqref{antisym}.) Clearly, in order to show that the constant terms of the original expressions are equal, it suffices to show this for the symmetrized expressions.  

The computation of the symmetrizer is usually much easier for the Magog-type constant term expressions, as $\prod_{1 \le i < j \le n} \frac{x_j-x_i}{1-x_i x_j}$ is antisymmetric and therefore it suffices to compute the \emph{antisymmetrizer}
of $\prod\limits_{i=1}^{n} f_i(x_i)$ (with $f_i(x)= (1+x_i)^i x_i^{-2i+2} (1-x_i^2)^{-1}$ in this case), which can often be accomplished using the Vandermonde determinant evaluation.

As for the Gog-type constant term expressions, the following result of Fonseca and Zinn-Justin \cite{FonZinn} can be used to compute the symmetrizer of \eqref{gog-zinnjustin} (and also to prove Theorem~\ref{zeil}). The symmetrization of the Gog-type expression \eqref{gog-zinnjustin} causes the ``core'' 
$\prod_{1 \le i < j \le n-1} (x_j-x_i)(1+x_j+x_i x_j)$ to be transformed into 
$\prod_{1 \le i < j \le n} \frac{x_j-x_i}{1-x_i x_j}$ (which is the ``core'' of the Magog-type expression \eqref{gog-zinnjustin}), and so the expressions are very similar (but not equal) after symmetrization.\footnote{The following identity is often helpful for further manipulations (the proof is left to the reader as we will not use it here):
Suppose $P(x_1,\ldots,x_n), S(x_1,\ldots,x_n)$ are two formal power series 
and $S(x_1,\ldots,x_n)$ is symmetric. Then the constant term of 
$$
\frac{S(x_1,\ldots,x_n) P(x_1,\ldots,x_n)}{(x_1 \cdots x_n)^{n-1}} \prod_{1 \le i < j \le n} (x_j-x_i) 
$$
agrees with the constant term of 
$$
\frac{S(0,\ldots,0) P(x_1,\ldots,x_n)}{(x_1 \cdots x_n)^{n-1}} \prod_{1 \le i < j \le n} (x_j-x_i). 
$$}
 
\begin{theo} 
\label{general}
Let 
$h_q(w,y)=(q w - q^{-1} y)(q w y - q^{-1}).$
Then the antisymmetrizer of 
\begin{equation}
\label{expr}
\frac{\prod_{1 \le i < j \le n} (q w_i - q^{-1} w_j)}{\prod_{1 \le i \le j \le n} h_1(w_j,y_i) \prod_{1 \le j \le i \le n} h_q(w_j,y_i)}
\end{equation}
w.r.t.\ $w_1,\ldots,w_n$ is 
$$
\frac{q^{\binom{n}{2}} \det_{1 \le i,j \le n} \left( 
\frac{1}{h_1(w_i,y_j) h_q(w_i,y_j)} \right)}{\prod_{1 \le i < j \le n} h_1(y_i,y_j)(1-q^2 w_i w_j)}.
$$
\end{theo}

Now in order to compute the constant term of \eqref{gog-zinnjustin}   using Theorem~\ref{general}, one sets $w_i=(x_i q^{-1} - 1)/(x_i q -1)$ in \eqref{expr} and obtains
\begin{multline}
\label{transform}
(q-q^{-1})^{\binom{n}{2}} \prod_{1 \le i < j \le n} (1 - (q+q^{-1}) x_j + x_i x_j) 
\prod_{i=1}^{n} (q x_i -1)^{n+3} \\
\times \left( \prod_{j=1}^{i}  (-1+y_j+x_i (q-q^{-1} y_j)) 
(1-y_j+x_i (-q^{-1} + q y_j)) \right. \\ \left. \times
\prod_{j=i}^{n} (-q^{-1}+ q y_j + x_i (1 -  y_j)) (q-q^{-1} y_j + x_i (-1+y_j)) \right)^{-1}.
\end{multline}
We need to compute the antisymmetrizer of 
\begin{equation}
\label{anti}
\prod_{i=1}^{n} x_i^{-2i+1} \prod_{1 \le i < j \le n} (1+ x_j + x_i x_j)
\end{equation}
After setting $y_j=1$ and $q=e^{2 i \pi/3}$, \eqref{transform} is up to a factor that is symmetric in $x_1,\ldots,x_n$ equal to \eqref{anti}. Using Theorem~\ref{general}, the antisymmetrizer can be written in terms of an expression involving a determinant, which can in turn be related to \eqref{magog-zinnjustin}.

\subsection{A variant of Theorem~\ref{general} and its application to  \eqref{constilse2} and to \eqref{alternative} when $b_i=i$}

Next we state a theorem that is similar to Theorem~\ref{general} and that can be proven analogously.

\begin{theo} 
\label{generalilse}
Let 
$h_q(w,y)=q w - q^{-1} y.$
Then the antisymmetrizer of 
\begin{equation}
\label{exprilse}
\frac{\prod_{1 \le i < j \le n} h_q(w_i,w_j)}{\prod_{1 \le i \le j \le n} h_1(w_j,y_i) \prod_{1 \le j \le i \le n} h_q(w_j,y_i)}
\end{equation}
w.r.t.\ $w_1,\ldots,w_n$ is 
$$
\frac{\det_{1 \le i,j \le n} \left( 
\frac{1}{h_1(w_i,y_j) h_q(w_i,y_j)} \right)}{\prod_{1 \le i < j \le n} h_1(y_j,y_i)}.
$$
\end{theo}

\begin{proof} The proof is by induction w.r.t.\ $n$. The case $n=1$ is easy to check.

Let 
$A(w_1,\ldots,w_n;y_1,\ldots,y_n)$ denote the antisymmetrizer of \eqref{exprilse}. We have the following recursion for $A(w_1,\ldots,w_n;y_1,\ldots,y_n)$.
$$
A(w_1,\ldots,w_n;y_1,\ldots,y_n) =
\sum_{k=1}^{n} (-1)^{k+n} \frac{\prod\limits_{1 \le i \le n, i \not= k} 
h_q(w_i,w_k)}{\prod\limits_{i=1}^{n} h_1(w_k,y_i) h_q(w_i,y_n)} 
A(w_1,\ldots,\widehat{w_k},\ldots,w_n;y_1,\ldots,y_{n-1})
$$ 
By the induction hypothesis, the right-hand side is equal to 
\begin{equation}
\label{rec1}
\sum_{k=1}^{n} (-1)^{k+n} \det\limits_{1 \le i \le n, i \not=k \atop 1 \le j \le n-1} \left( 
\frac{1}{h_1(w_i,y_j) h_q(w_i,y_j)} \right) \frac{\prod\limits_{1 \le i \le n, i \not= k} 
h_q(w_i,w_k) \prod\limits_{1 \le i \le n-1} h_1(y_n,y_i)}{\prod\limits_{i=1}^{n} h_1(w_k,y_i) h_q(w_i,y_n) \prod\limits_{1 \le i < j \le n} h_1(y_j,y_n)}.
\end{equation}
We define 
$$
g_l(w_1,\ldots,w_n;y_1,\ldots,y_n) = \prod_{1 \le i \le n-1, i \not= l} 
\frac{h_1(y_i,y_n)}{h_1(y_l,y_i)} \prod_{i=1}^{n} 
\frac{h_q(w_i,y_l)}{h_q(w_i,y_n)}
$$
and claim that 
\begin{equation}
\label{crucial}
\sum_{l=1}^{n} g_l(w_1,\ldots,w_n;y_1,\ldots,y_n) 
\frac{1}{h_1(w_k,y_l) h_q(w_k,y_l)} = 
 \frac{\prod\limits_{1 \le i \le n, i \not=k} h_q(w_i,w_k) \prod\limits_{1 \le i \le n-1} h_1(y_i,y_n)}
{\prod\limits_{1 \le i \le n} h_1(w_k,y_i) h_q(w_i,y_n)}.
\end{equation}
Assuming \eqref{crucial} is true, we can replace in \eqref{rec1} the expression on the right-hand side of \eqref{crucial} by the left-hand side of 
\eqref{crucial}. We change the order of summation and obtain 
$$
(-1)^{n-1} \prod_{1 \le i < j \le n} h_1(y_j,y_i)^{-1} \sum_{l=1}^{n} g_l \left. \det_{1 \le i, j \le n} \left( \frac{1}{h_1(w_i,y_j) h_q(w_i,y_j)} \right) \right|_{y_n=y_l} 
$$
and the last expression is equal to the expression in the theorem.
In order to prove \eqref{crucial}, we rearrange the identity and obtain 
$$
\sum_{l=1}^{n} \prod_{1 \le i \le n-1 \atop i \not= l} \frac{h_1(y_i,y_n)}{h_1(y_l,y_i)} 
\prod_{1 \le i \le n \atop i \not= k} h_q(w_i,y_l) \prod_{1 \le i \le n \atop i \not= l} h_1(w_k,y_i) \\
= \prod_{1 \le i \le n, i \not=k} h_q(w_i,w_k) \prod_{1 \le i \le n-1} h_1(y_i,y_n).
$$
We consider both sides as polynomials in $w_k$. The degree is in both cases not greater than $n-1$ and so it suffices to show that they agree at the evaluations $w_k=y_p$, $1 \le p \le n$. In this case, each summand on the left-hand side vanishes except for the summand corresponding to $l=p$. 
\end{proof}
Note that the statement of this theorems differs from the statement of Theorem~\ref{general} only in so far that we set $h_q(w,y)=q w - q^{-1} y$ instead of 
$h_q(w,y)=(q w - q^{-1} y)(q w y - q^{-1})$.

Before we apply this theorem, we mention another formula that will be useful in the following.  The formula appeared in \cite[Eq (43)-(47)]{BehrendWeightedEnum}. Suppose $f(x,y)$ is a power series in $x$ and $y$, then 
\begin{equation}
\label{behrend}
\lim_{(x_1,\ldots,x_n) \to (x,\ldots,x) \atop (y_1,\ldots,y_n) \to (y,\ldots,y)}
\frac{\det\limits_{1 \le i, j \le n} \left( f(x_i,y_j) \right)}{\prod\limits_{1 \le i < j \le n} (x_j-x_i)(y_j-y_i)} = \det_{0 \le i,j  \le n-1} \left( [u^i v^j] f(x+u,y+v) \right),
\end{equation}
where $[u^i v^j] f(x+u,y+v)$ denotes the coefficient of $u^i v^j$ in $f(x+u,y+v)$. A close relative of this formula is the following: Suppose $f_j(x)$ is a power series in $x$ for $1 \le j \le n$, then 
\begin{equation}
\label{behrendrelative}
\lim_{(x_1,\ldots,x_n) \to (x,\ldots,x)} \frac{\det\limits_{1 \le i,j \le n} \left( f_j(x_i) \right)}{\prod\limits_{1 \le i < j \le n} (x_j-x_i)} = \det\limits_{0 \le i \le n-1 \atop 1 \le j \le n } \left( [u^i] f_j(x+u) \right).
\end{equation}

We compute the symmetrizer of \eqref{constilse2}: We set $w_i=(x_i+1+q^{-1})/(x_i+1+q)$ in \eqref{exprilse} and obtain 
\begin{multline*}
\left(q - q^{-1} \right)^{\binom{n}{2}} \prod_{1 \le i < j \le n} 
(q^{-1}+2+q+(q^{-1}+1+q) x_i + x_j + x_i x_j) \prod_{i=1}^{n} (1+q+x_i)^2 \\
\times \left( \prod_{i=1}^{n} \prod_{j=1}^{i} 
\left(x_i (1-y_j) + (1+q^{-1})(1-q y_j) \right)
\prod_{j=i}^{n} \left( x_i (q - q^{-1} y_j) + (1+q)(1- y_j q^{-1}) \right) \right)^{-1}.
\end{multline*} 
By setting $y_i=1$, $q=e^{2 i \pi/3}$ and applying Theorem~\ref{generalilse}, we obtain 
\begin{multline}
\label{twice}
{\mathcal AS}_{x_1,\ldots,x_n} \left[ \prod_{1 \le i < j \le n} (1+x_j + x_i x_j) \prod_{i=1}^n (1+x_i)^i x_i^{-n+1} \right] \\
=(-1)^{\binom{n+1}{2}} q^n (q-q^{-1})^{(n+3)n/2}  \prod_{i=1}^n (1+x_i)^{n+1} x_i^{-n+1} (1+q+x_i)^{-2} \\
\times \lim_{(y_1,\ldots,y_n) \to 1} \det_{1 \le i,j \le n} \left( \frac{1}{\left(y_j - \frac{x_i + 1 + q^{-1}}{x_i+1+q} \right)
\left(y_j - q^2 \frac{x_i + 1 + q^{-1}}{x_i+1+q}\right)}  \right) \prod_{1 \le i < j \le n} (y_j-y_i)^{-1}.
\end{multline}
We use the following partial fraction decomposition 
$$
\frac{1}{(y-a)(y-b)} = \frac{1}{a-b} \left( \frac{1}{y-a} - \frac{1}{y-b} \right)
$$
to rewrite the matrix entry of the determinant.
We obtain 
\begin{multline}
\label{2ncauchy}
(-1)^{\binom{n}{2}} (q-q^{-1})^{\binom{n+1}{2}}  \prod_{i=1}^n (1+x_i)^{n+1} x_i^{-n+1} (x_i+1+q)^{-1} (x_i+1+q^{-1})^{-1} \\
\times \lim_{(y_1,\ldots,y_n) \to 1} \det_{1 \le i,j \le n} \left( \frac{1}{\left(y_j - \frac{x_i + 1 + q^{-1}}{x_i+1+q} \right)}
-\frac{1}{\left(y_j - q^2 \frac{x_i + 1 + q^{-1}}{x_i+1+q}\right)}  \right) \prod_{1 \le i < j \le n} (y_j-y_i)^{-1}
\end{multline}
for the right-hand side. Now we can apply \eqref{behrendrelative} to obtain the following.
$$
\prod_{i=1}^n (1+x_i)^{n+1} x_i^{-n+1} (x_i+1+q^{-1})^{-1}  \det_{1 \le i,j \le n} \left( (x_i+1+q)^{j-1}  \left( 1- (-1-x_i)^{-j} q^{-j}  \right) \right)
$$

Alternatively, we can also use the Cauchy determinant.
$$
\det_{1 \le i, j \le n} \left( \frac{1}{x_i + y_j} \right) = \frac{\prod_{1 \le i < j \le n} (x_j-x_i)(y_j-y_i)}{\prod_{i,j=1}^{n} (x_i+y_j)}
$$
The determinant in \eqref{2ncauchy} can be written as a sum of $2^n$ Cauchy determinants.
\begin{multline*}
\det_{1 \le i,j \le n} \left( \frac{1}{\left(y_j - \frac{x_i + 1 + q^{-1}}{x_i+1+q} \right)}
-\frac{1}{\left(y_j - q^2 \frac{x_i + 1 + q^{-1}}{x_i+1+q}\right)}  \right)  =
\sum_{(s_1,\ldots,s_n) \in \{0,1\}^n} \det_{1 \le i,j \le n} \left( (-1)^{s_i} \frac{1}{\left(y_j - q^{2 s_i} \frac{x_i + 1 + q^{-1}}{x_i+1+q} \right)} \right)
\end{multline*}
In a similar situation, namely the proof of the equality of the constant terms of 
\eqref{gog-zinnjustin} and \eqref{magog-zinnjustin}, only one of these Cauchy determinants contributes to the constant term. This is not true here, which is mainly due to the factor $\prod\limits_{i=1}^n x_i^{-n+1}$ in \eqref{constilse2}. However, it is still possible to get rid of the extra set of variables $y_1,\ldots,y_n$ and it follows that the right-hand side is equal to 
\begin{multline*}
 \prod_{i=1}^n (1+x_i)^{n+1} x_i^{-n+1} (x_i+1+q^{-1})^{-1} 
\sum_{(s_1,\ldots,s_n) \in \{0,1\}^n} (-1)^{(n-1)(s_1+\ldots+s_n)+\binom{n}{2}} q^{-s_1-\ldots-s_n} \prod_{i=1}^{n} (1+x_i)^{- n s_i} \\ 
\times \prod_{1 \le i < j \le n} \left((1+  x_i x_j) (s_i-s_j) + x_i (1-s_i - 2 s_j + 3 s_i s_j) + x_j (-1+s_j+2 s_i - 3 s_i s_j)  \right), 
\end{multline*}
and this can also be written as follows.
$$
\prod_{i=1}^n (1+x_i)^{n+1} x_i^{-n+1} (x_i+1+q^{-1})^{-1} \det_{1 \le i, j \le n}
\left( x_i^{j-1}-(-1-x_i)^{-j} q^{-1} \right)
$$

\medskip

As for computing the symmetrizer of \eqref{alternative} when $b_i=i$,
we set $w_i=(x_i q^{-1} - 1)/(x_i q -1)$ in \eqref{exprilse} and obtain
\begin{multline}
\label{transformilse}
(q-q^{-1})^{\binom{n}{2}}
\prod_{1 \le i < j \le n} (1 - (q+q^{-1}) x_j + x_i x_j) 
\prod_{i=1}^{n} (q x_i -1)^{2} \\
\times
\left( \prod_{j=1}^{i} (-1+y_j + x_i (q^{-1} - q y_j))
 \prod_{j=i}^{n} (-q+y_j q^{-1} + 
x_i (1 - y_j)) \right)^{-1}
\end{multline}
Furthermore, we set $y_i=1$ and $q=e^{i \pi/3}$, and Theorem~\ref{generalilse} now implies in a similar way as above
\begin{multline*}
{\mathcal AS}_{x_1,\ldots,x_n} \left[ \prod_{1 \le i < j \le n} (1-x_j + x_i x_j) \prod_{i=1}^n (1-x_i)^{-n} x_i^{-n+1-i} \right]
= \\
(-1)^n (q-q^{-1})^{(n+1)n/2}  \prod_{i=1}^n (1-x_i)^{-n} x_i^{-n+1} (q x_i-1)^{-1}
(q^{-1} x_i-1)^{-1} \\
\times \lim_{(y_1,\ldots,y_n) \to 1} \det_{1 \le i,j \le n} \left( \frac{1}{\left(y_j - \frac{x_i q^{-1} -1}{x_i q -1} \right)}-\frac{1}{
\left(y_j - q^2 \frac{x_i q^{-1} -1}{x_i q -1} \right)}  \right) \prod_{1 \le i < j \le n} (y_j-y_i)^{-1}.
\end{multline*}
Using \eqref{behrendrelative}, we obtain 
$$
 \prod_{i=1}^n (1-x_i)^{-n} x_i^{-n+1} 
(q^{-1} x_i-1)^{-1}  \det_{1 \le i,j \le n} \left( (1-x_i q)^{j-1} 
\left( q^{-j} - x_i^{-j} \right) \right).
$$
Alternatively, we can also use the approach from above involving the Cauchy determinant and obtain
\begin{multline*}
  \prod_{i=1}^n (1-x_i)^{-n} x_i^{-n+1} 
(q^{-1} x_i-1)^{-1} 
\sum_{(s_1,\ldots,s_n) \in \{0,1\}^n} (-q)^{-s_1-\ldots-s_n} (-1)^{\binom{s_1+\ldots+s_n}{2}+\binom{n-s_1-\ldots-s_n}{2}+n}
\prod_{i=1}^{n} x_i^{n(s_i-1)} \\ 
\times \prod_{1 \le i < j \le n} \left((1+  x_i x_j) (s_i-s_j) + x_i (-1+s_i + 2 s_j - 3 s_i s_j) + x_j (1-s_j-2 s_i + 3 s_i s_j)  \right).
\end{multline*}
This can also be written as
$$
\prod_{i=1}^n (1-x_i)^{-n} x_i^{-n+1} (q^{-1} x_i-1)^{-1} \det_{1 \le i,j \le n} 
\left(-x_i^{-j} + (1-x_i)^{j-1} q^{-1} \right). 
$$

\subsection{The application of Theorem~\ref{generalilse} to \eqref{alternative2} when $b_i=i$} The expression in \eqref{alternative2} is more complicated at first glance as it already involves the antisymmetrizer operator. However, as our approach involves the computation of the antisymmetrizer anyway, this is no disadvantage.  

By \eqref{twice}, \eqref{alternative2} is equal to 
\begin{multline}
\label{third}
(-1)^{\binom{n+1}{2}} q^n (q-q^{-1})^{(n+3)n/2}  \prod_{i=1}^n (1+x_i)^{n+1} (1+q+x_i)^{-2} \\
\times \lim_{(y_1,\ldots,y_n) \to 1} \det_{1 \le i,j \le n} \left( \frac{1}{\left(y_j - \frac{x_i + 1 + q^{-1}}{x_i+1+q} \right)
\left(y_j - q^2 \frac{x_i + 1 + q^{-1}}{x_i+1+q}\right)}  \right) \prod_{1 \le i < j \le n} (x_j-x_i)^{-1} (y_j-y_i)^{-1}, 
\end{multline}
where $q=e^{2 \pi i /3}$.
The same procedure as above can be used to show that this is equal to 
\begin{multline}
\label{single}
 \prod_{i=1}^n (1+x_i)^{n+1} (x_i+1+q^{-1})^{-1} 
\sum_{(s_1,\ldots,s_n) \in \{0,1\}^n} (-1)^{(n-1)(s_1+\ldots+s_n)+\binom{n}{2}} q^{-s_1-\ldots-s_n} \prod_{i=1}^{n} (1+x_i)^{- n s_i} \\ 
\times \prod_{1 \le i < j \le n} \frac{(1+  x_i x_j) (s_i-s_j) + x_i (1-s_i - 2 s_j + 3 s_i s_j) + x_j (-1+s_j+2 s_i - 3 s_i s_j)}{x_j-x_i} \\
= \prod_{i=1}^n (1+x_i)^{n+1} (x_i+1+q^{-1})^{-1}  \det_{1 \le i, j \le n}
\left( x_i^{j-1}-(-1-x_i)^{-j} q^{-1} \right)\prod_{1 \le i < j \le n} (x_j-x_i)^{-1}.
\end{multline}

On the other hand, the advantage of \eqref{alternative2} is that it is actually a polynomial and therefore we can compute the constant term using \eqref{behrend}. Now, by \eqref{third}, it follows that the number of $n \times n$ ASMs is 
\begin{multline}
\lim_{(x_1,\ldots,x_n) \to 0 \atop (y_1,\ldots,y_n) \to 0}
 (-1)^n (q^{-1}-q)^{\binom{n+1}{2}}  \prod_{i=1}^n (1+x_i)^{n+1} (1+q^{-1}+x_i)^{-1}
\prod_{1 \le i < j \le n} (x_j-x_i)^{-1} (y_j-y_i)^{-1}  \\
\times  \det_{1 \le i,j \le n} \left( \frac{1}{(y_j +1)(x_i+1+q) - x_i - 1 - q^{-1}}
- \frac{1}{(y_j +1)(x_i+1+q) - q^2 (x_i + 1 + q^{-1})} \right).
\end{multline}
The coefficient of $x^i y^j$ in
$$
\frac{1}{(y +1)(x+1+q) - x - 1 - q^{-1}}
- \frac{1}{(y +1)(x+1+q) - q^2 (x + 1 + q^{-1})}
$$
when considering this expression as a power series in $x,y$ is
$$
\binom{j}{i} (1+q)^{-i-1} (q-1)^{-j-1} (-1)^j q^{j+1} -
\sum_{k=i}^{i+j} \binom{k}{j} \binom{j}{i-k+j} \frac{(1+q)^{k-i-j-1}}{(1-q)^{j+1}} 
(-1)^k.
$$
It follows that the number of $n \times n$ ASMs is given by 
$$
(1+q)^{-n} (-q)^{-\binom{n}{2}} \det_{0 \le i, j \le n-1} \left( 
\binom{j}{i} (-1)^j q^{j+1} + \sum_{k=0}^{n-1} \binom{k+i}{j} \binom{j}{k} (-1-q)^{k+i-j} \right).
$$
Using basic properties of the binomial coefficient and the Chu-Vandermonde summation, it can be shown that 
\begin{multline*}
\binom{j}{i} (-1)^j q^{j+1} + \sum_{k=0}^{n-1} \binom{k+i}{j} \binom{j}{k} (-1-q)^{k+i-j} \\
= \sum_{k=0}^{n-1} (-1)^i \binom{i}{k} \binom{-k-1}{j} \left( q^{j+1} (-1)^k + q^k (-1)^j \right), 
\end{multline*}
and it follows that the number of $n \times n$ ASMs is given by 
\begin{multline*}
\det \left[ \left( 
\binom{i}{j} (-1)^{i+j} \right)_{0 \le i,j \le n-1} \cdot \left( 
\binom{i+j}{j} \frac{1 - (-q)^{j+1-i}}{1+q} \right)_{0 \le i,j \le n-1} \right] \\=
\det_{0 \le i,j \le n-1}  \left(\binom{i+j}{j} \frac{1 - (-q)^{j+1-i}}{1+q} \right),
\end{multline*}
provided that $q=e^{2 i \pi / 3}$. This seems to be a new determinant for the ASM numbers.

In a forthcoming paper by F. Aigner, the more general determinant 
\begin{equation}
\label{parameter}
\det_{0 \le i, j \le n-1} \left( 
\binom{x+i+j}{j} \frac{1 - (-q)^{j+1-i}}{1+q} \right)
\end{equation}
will be computed for all sixth roots of unity $q$ not equal to $1$, and consequences of this for the enumeration of Gogs.
In these cases, the determinants have only integer zeros as a polynomial in $x$. 
Moreover, the determinant is related to the following determinant 
\begin{equation}
\label{zare}
\det_{0 \le i, j \le n-1} \left( 
\binom{x+i+j}{j} + q \, \delta_{i,j} \right)
\end{equation}
that was considered by Ciucu, Eisenk\"olbl, Krattenthaler and Zare \cite{zare}, and computed for all sixth roots of unity $q$. To be more precise, the quotient of the first and the second determinant is 
$(-q)^{n}$ if $q=-\frac{1}{2}+\frac{\sqrt{3}}{2} i$, and it is $(-q)^{-n}$ if $q=-\frac{1}{2}-\frac{\sqrt{3}}{2} i$.
Other curious observations such as the following will be studied: For general $q$, there seems to be a sequence of functions $p_n(x,q)$ that are polynomials 
in $x$ of degree no greater than $\binom{\lceil (n+1)/2 \rceil}{2}$ and Laurent polynomials in $q$ with highest exponent 
$\binom{\lfloor (n+1)/2 \rfloor}{2}$ and lowest exponent $-\binom{\lfloor (n+1)/2 \rfloor}{2}$ over $\mathbb{Q}$ such that the determinant in \eqref{parameter} is equal to $p_{n-1}(x,q) p_{n}(x,q)$.

Let us finally remark that one could of course apply \eqref{behrendrelative} to \eqref{single} and obtain 
$$
\det_{0 \le i,j \le n-1} \left(- q \binom{i+j}{i}  - q^2 \delta_{i,j} \right),
$$
which is up to a trivial factor a specialization of \eqref{zare} and has thus already been computed.

\section{Acknowledgement}  The author thanks an anonymous referee for the careful reading of the paper and several interesting comments.

\appendix

\section{The case $v=1-u$}
\label{v1minusu}
\subsection{The cases $(u,v)=(0,1)$ and $(u,v)=(1,0)$} The generating function of monotone triangles with bottom row $b_1,\ldots,b_n$ evaluated at $(u,v)=(0,1)$ is equal to the number of monotone triangles with that bottom row, where SE-diagonals are strictly increasing, while the evaluation at $(u,v)=(1,0)$ is the number of monotone triangles that are strictly increasing along NE-diagonals. The number is in both cases $\prod_{1 \le i < j \le n} \frac{b_j-b_i}{j-i}$ as both sets are in bijective correspondence with Gelfand-Tsetlin patterns with bottom row 
$b_1-1,b_2-2,\ldots,b_n-n$ (in the first case, this follows by subtracting $i$ from the $i$-th NE-diagonal for all $i$, counted from the left, while in the second case, this follows by subtracting $i$ from the $i$-th SE-diagonal for all $i$, also counted from the left), and the number of Gelfand-Tsetlin patterns with bottom row $b_1,\ldots,b_n$ is $\prod_{1 \le i < j \le n} \frac{b_j-b_i+j-i}{j-i}$, see \cite{gelfand} or \cite[Corollary 7.21.4 and Lemma
7.21.1]{Stanley2}.

\subsection{The case $(u,v)=(\frac{1}{2},\frac{1}{2})$} On the other hand, the generating function of monotone triangles with bottom row $b_1,\ldots,b_n$ evaluated at $(u,v)=(\frac{1}{2},\frac{1}{2})$ is---up to the factor $2^{\binom{n}{2}}$---equal to the $2$-enumeration of monotone triangles with respect to the number of entries that lie strictly between their SW-neighbors and 
and their SE-neighbors.  By \cite[Theorem 2]{DPPMRR}, this number is $2^{\binom{n}{2}} \prod_{1 \le i < j \le n} \frac{b_j-b_i}{j-i}$. If $(b_1,\ldots,b_n)=(1,\ldots,n)$, this corresponds to the $2$-enumeration of ASMs with respect to the number of $-1$'s, or, equivalently, to the enumeration of domino tilings of the Aztec diamond of order $n$, see \cite{ProppDomino} and also \cite[Remark 4.3]{CiucuReflect}, where in order to obtain the number of the latter one has to multiply by $2^n$. 

\subsection{The general case} The results mentioned in the previous two paragraphs can also be deduced from Theorem~\ref{operator} as follows. In all cases, 
$\prod_{1 \le p < q \le n} \e_{x_q} \st_{x_q,x_p}$ is a symmetric polynomial in $\fd_{x_1},\ldots,\fd_{x_n}$ with constant term $1$, and the enumeration formulas follow from the fact that the application of an ``operator'' polynomial with these properties to
$\prod_{1 \le i < j \le n} \frac{x_j-x_i}{j-i}$ leaves the polynomial invariant as was shown, e.g., in \cite[Lemma 2.5]{FischerASMProof2015}. In fact, a common generalization follows if we assume that 
$v=1-u$, because in this case the same argument applies and we can conclude that the generating function is $\prod_{1 \le i < j \le n} \frac{b_j-b_j}{j-i}$. (This result could also be derived using known proofs for the $2$-enumeration of ASMs or the enumeration of perfect matchings of the Aztec diamonds, see, e.g.,  \cite{DPPMRR,ProppDomino}.) In particular, it follows that the generating function is independent of $u$. 

\subsubsection{Perfect matchings}
The specialization $v=1-u$ of the generating function 
is also the weighted enumeration of perfect matchings of a certain portion of 
the square grid: First observe that, by extending the bijection between monotone triangles 
with bottom row $1,2,\ldots,n$ and $n \times n$ ASMs to monotone triangles 
with arbitrary increasing bottom row of positive\footnote{It is no restriction 
to confine our considerations to monotone triangles that contain only positive 
integers in the bottom row, because every monotone triangle that contains 
non-positive integers can be transformed to one that has only positive 
integers by adding the same positive integer to every entry.} integers, 
say, $b_1,\ldots, b_n$, we see that the latter are in bijection to $n \times m$ 
matrices, where $m$ is any positive integer with $m \ge b_n$ that have the same 
properties as ASMs with the exception that column sums do not have to be $1$, 
but it is still required that the topmost non-zero entry of each column is $1$, 
and, in addition, the column sums have to be $1$ precisely for the columns 
$b_1,\ldots,b_n$.

\begin{figure}
\scalebox{0.3}{
\raisebox{2cm}
{\includegraphics{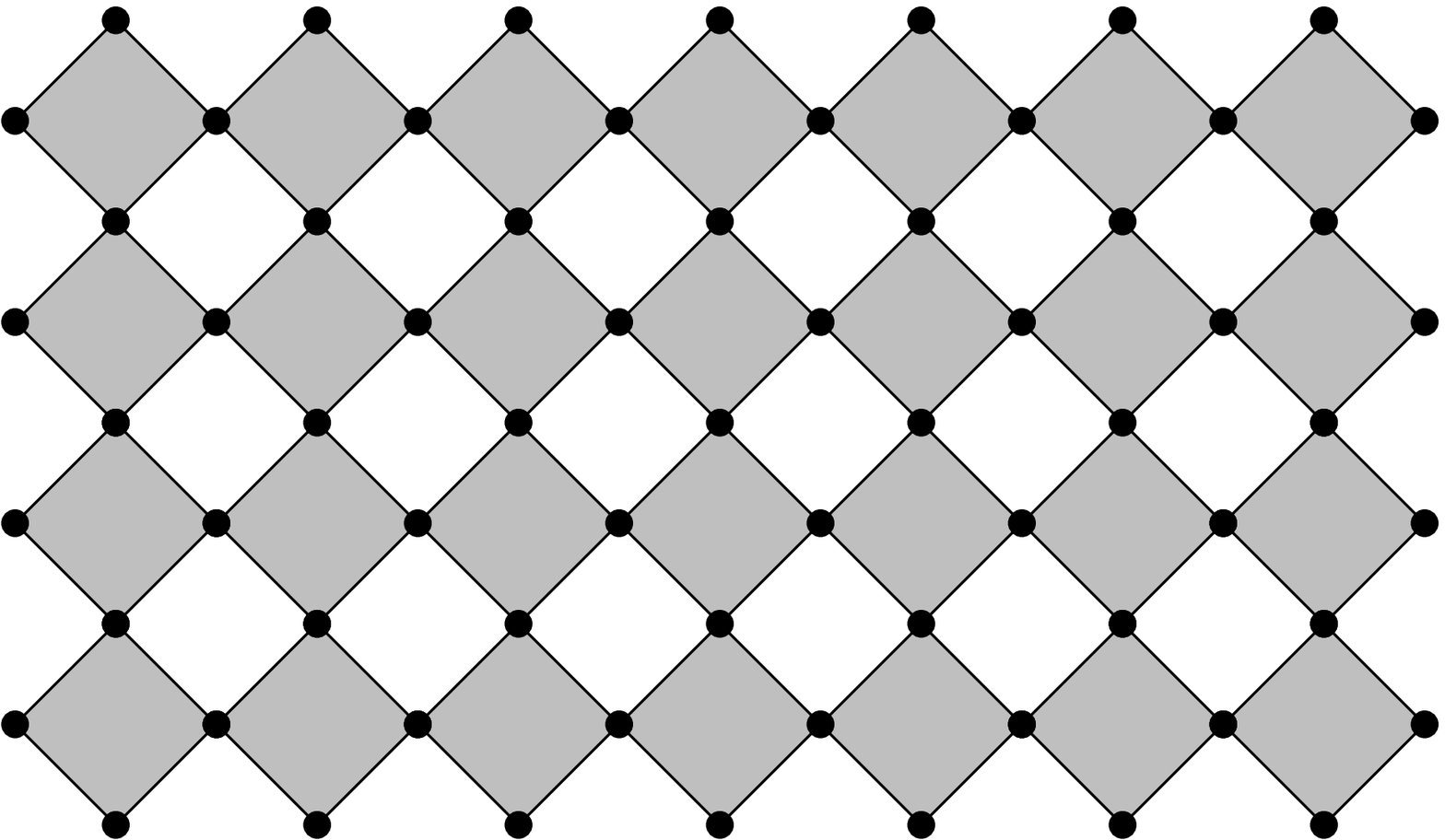}}} \qquad \qquad
\scalebox{0.3}{
\includegraphics{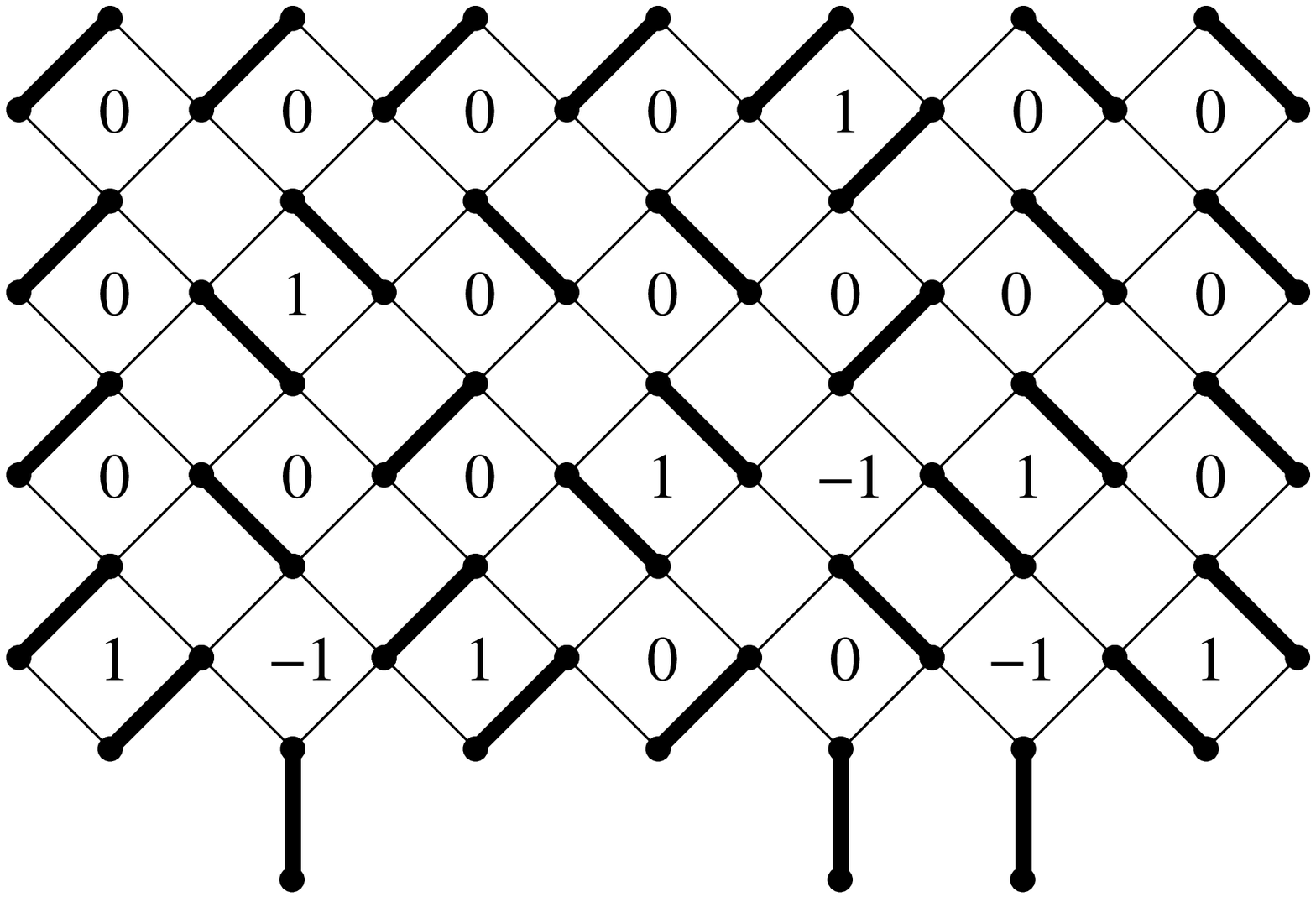}}
\caption{\label{ar} $\ar^{4,7}$ (left) and an example (right)}
\end{figure}

Now let $\ar^{n,m}$ denote the rectangular graph that consists 
of $n$ rows of sequences of $m$ consecutive cells of the form $\Diamond$, see Figure~\ref{ar} (left) for $\ar^{4,7}$. It is a well-known fact, see \cite{DissMihai}, that the perfect matchings of $\ar^{n,n}$ can be partitioned into classes that are indexed by $n \times n$ ASMs, such that, for a given ASM, the number of perfect matchings in the associated class is a power of $2$, where the exponent is just the number of $1$'s in the ASMs. This can be extended to the ``rectangular alternating sign matrices'' described in the previous paragraph as follows. For a given strictly increasing sequence of positive integers $\mathbf{b}=(b_1,\ldots,b_n)$ of length $n$ and $m \ge b_n$, let $\ar^{n,m}_{\mathbf{b}}$ denote the graph that is obtained from $\ar^{n,m}$ by adding vertical edges incident with the bottom vertices of the cells in the bottom row of $\ar^{n,m}$ except for the vertices in columns $b_1,\ldots,b_n$. For instance, the underlying graph in Figure~\ref{ar} (right) is $\ar^{4,7}_{1,3,4,7}$. Then the perfect matchings of $\ar^{n,m}_{\mathbf{b}}$ can be partitioned such that the classes are indexed by monotone triangles with bottom row $\mathbf{b}$ and the number of elements in each class is a power of $2$, where the exponent is the number of $1$'s in the corresponding rectangular alternating sign matrix. The perfect matching in 
Figure~\ref{ar} (right) is a perfect matching that lies in the class of the following monotone 
triangle.
$$
\begin{array}{ccccccc}
& & & 5 & & & \\
& & 2 & & 5 & & \\
& 2 & & 4 & & 6 \\
1 & & 3 & & 4 & & 7
\end{array} 
$$
The corresponding rectangular alternating sign matrix can be obtained by counting, for each cell, the number of edges that are part of the perfect matching and subtract $1$, see 
Figure~\ref{ar} (right), where we have put the numbers into the appropriate cells. The other perfect matchings that are in the class of the given monotone triangle are obtained by ``rotating''  the perfect matching edges of those cells thatcontain two matching edges independently.

Next we introduce edge weights such that the weighted enumeration of the perfect matchings of $\ar^{n,m}_{\mathbf{b}}$ associated with a fixed monotone triangle $M$ is just $u^{\inv(M)} (1-u)^{\inv'(M)}$: in each cell, we assign the weights $1,1,u,1-u$ to the edges, where we start at the NW edge and go around the cell clockwise, see Figure~\ref{urbanrenewal} right.
As a side remark note that this shows that the special cases $(u,v)=(1,0)$ and 
$(u,v)=(0,1)$ in the generating function amounts to compute the number of perfect matchings of a \emph{hexagonal grid} (since $(u,v)=(1,0)$ corresponds to the deletion of 
the SW edge of each cell, while $(u,v)=(0,1)$ corresponds to the deletion of the 
SE edge of each cell), which is no surprise, because counting perfect matchings of hexagonal grids corresponds to \emph{lozenge tiling enumeration}, which in turn is related to the enumeration of \emph{semistandard tableaux} and thus of Gelfand-Tsetlin patterns.

\begin{figure}
\scalebox{0.4}{
\psfrag{a}{\Large $a$}
\psfrag{b}{\Large$b$}
\psfrag{c}{\Large$c$}
\psfrag{d}{\Large$d$}
\psfrag{A}{\Large $\frac{a}{ad+bc}$}
\psfrag{B}{\Large$\frac{b}{ad+bc}$}
\psfrag{C}{\Large$\frac{c}{ad+bc}$}
\psfrag{D}{\Large$\frac{d}{ad+bc}$}
\includegraphics{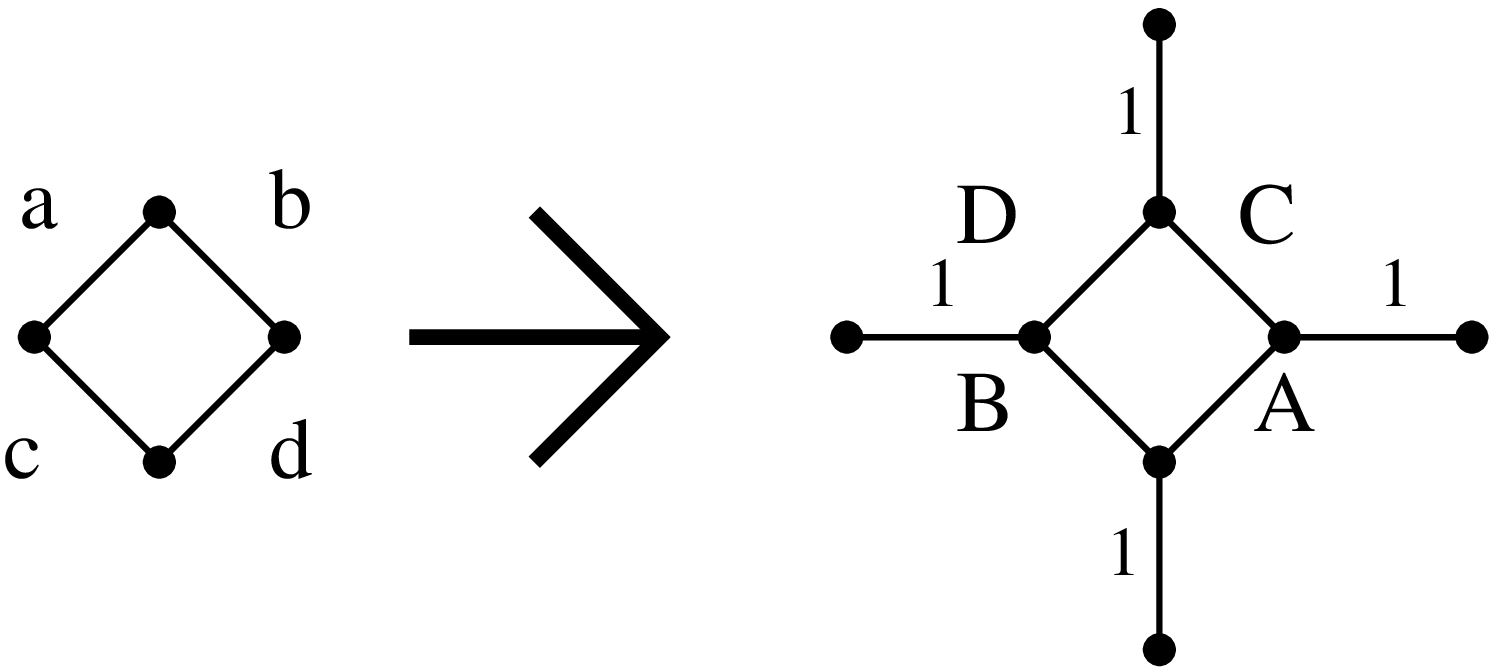}} \quad \quad \quad
\scalebox{0.4}{
\psfrag{1}{\Large $1$}
\psfrag{u}{\Large$u$}
\psfrag{1-u}{\Large$1-u$}
\includegraphics{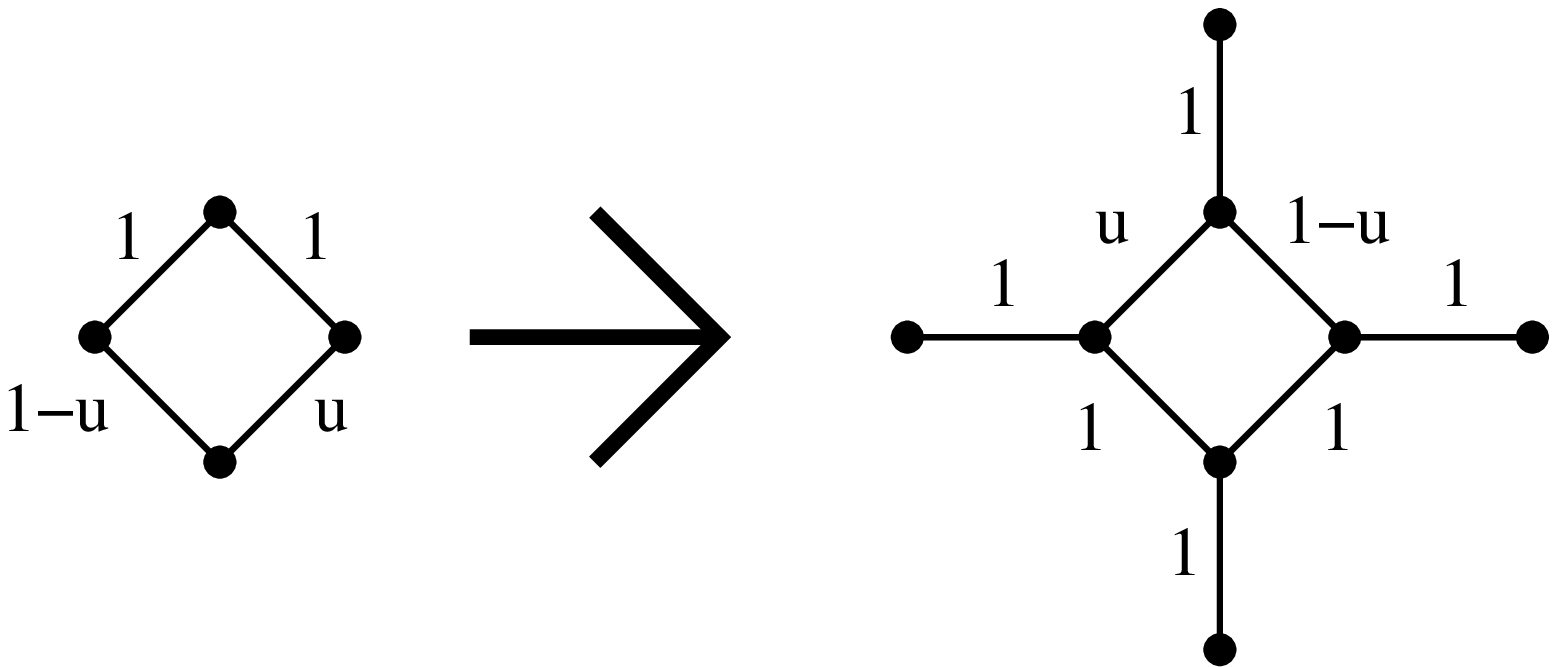}}\caption{\label{urbanrenewal} Urban renewal}\end{figure}

\subsubsection{Independence of $u$} We sketch an argument that shows that the weighted enumeration of 
our perfect matchings is independent of $u$, and thus the generating function of monotone triangles (after setting $v=1-u$): This is done with the help of the local graph operation \emph{urban renewal} which was introduced by Kuperberg and Propp, see 
Figure~\ref{urbanrenewal}. If we replace a cell with edge weights $a,b,c,d$ by the configuration indicated in Figure~\ref{urbanrenewal} (right), the generating function of perfect matchings of the original graph is obtained from the generating function of the modified graph by multiplication of $a d + b c$. We perform this operation to every cell of, say, $\ar^{4,7}_{1,3,4,7}$. 

In the graph that is obtained this way, there are two trivial simplifications that can be made at several places: If the two edges that are incident with a vertex of degree $2$ have weight $1$, they can be contracted without changing the generating function, and each edge with weight $1$ that is incident with a vertex of degree $1$ can be deleted along with all edges incident with the other vertex of the edge. This implies that the weighted enumeration of the perfect matchings of $\ar^{4,7}_{1,3,4,7}$ is equal to the sum of the weighted enumeration of the perfect matchings of the 
following graphs: $\ar^{3,6}_{2,3,6}$, $\ar^{3,6}_{1,3,6}$, $\ar^{3,6}_{2,3,5}$, 
$\ar^{3,6}_{1,3,5}$, $\ar^{3,6}_{2,3,4}$, $\ar^{3,6}_{1,3,4}$. (The sequences arise as follows: The complement of $\{1,3,4,7\}$ in $\{1,\ldots,7\}$ is $\{2,5,6\}$, and now we allow for each element $i$ in the complement that either the element itself or $i-1$ is an element of the new complements. If we take complements in $\{1,\ldots,6\}$, then we obtain our sequences.) We can assume by induction with respect to $n$ that the generating function of the perfect matchings of each of these graphs is independent of $u$.

\end{document}